\documentclass[reqno]{amsart}
\usepackage[margin=1.5in]{geometry}
\usepackage{amssymb}
\usepackage{comment}
\usepackage{etex}
\usepackage{etoolbox}
\usepackage{mathtools}
\usepackage{stmaryrd}
\usepackage{tikz-cd}
\usepackage[textsize=scriptsize]{todonotes}
\usepackage{verbatim}

\usepackage[T1]{fontenc}
\usepackage{hyperref}
\usepackage[noabbrev,nameinlink]{cleveref}

\hypersetup{
    colorlinks,
    linkcolor={red!30!black},
    citecolor={blue!50!black},
    urlcolor={blue!80!black}
}

\usepackage{relsize,scalerel}
\newcommand{\bigast}{\mathop{\scalebox{2.5}{\raisebox{-0.35ex}{$\ast$}}}}
\let\oldbigcirc\bigcirc
\def\bigcirc{\mathop{\scalerel*{\oldbigcirc}{\bigast}}}

\let\oldtocsection=\tocsection

\let\oldtocsubsection=\tocsubsection

\let\oldtocsubsubsection=\tocsubsubsection

\renewcommand{\tocsection}[2]{\hspace{0em}\oldtocsection{#1}{#2}}
\renewcommand{\tocsubsection}[2]{\hspace{1em}\oldtocsubsection{#1}{#2}}
\renewcommand{\tocsubsubsection}[2]{\hspace{2em}\oldtocsubsubsection{#1}{#2}}

\theoremstyle{definition}
\newtheorem{nul}{}[section]
\newtheorem{dfn}[nul]{Definition}

\newtheorem{rmk}[nul]{Remark}

\newtheorem{cnstr}[nul]{Construction}
\newtheorem{cnv}[nul]{Convention}
\newtheorem{ntn}[nul]{Notation}
\newtheorem{exm}[nul]{Example}

\newtheorem{rec}[nul]{Recollection}

\newtheorem{qst}{Open Question}
\Crefname{qst}{Question}{Questions}
\newtheorem*{dfn*}{Definition}
\newtheorem*{axm*}{Axiom}
\newtheorem*{ntn*}{Notation}
\newtheorem*{exm*}{Example}
\newtheorem*{exr*}{Exercise}
\newtheorem*{form*}{Formula}
\newtheorem*{int*}{Intuition}
\newtheorem*{qst*}{Question}
\newtheorem*{rmk*}{Remark}
\newtheorem*{cnstr*}{Construction}
\newtheorem*{cnvs*}{Conventions}

\theoremstyle{plain}

\newtheorem{thm}[nul]{Theorem}
\newtheorem{prop}[nul]{Proposition}
\newtheorem{lem}[nul]{Lemma}

\newtheorem{cor}[nul]{Corollary}
\newtheorem*{thm*}{Theorem}
\newtheorem*{prop*}{Proposition}
\newtheorem*{cor*}{Corollary}
\newtheorem*{lem*}{Lemma}
\newtheorem*{cnj*}{Conjecture}

\DeclareMathOperator*{\colim}{\mathrm{colim}}
\DeclareMathOperator{\Map}{\mathrm{Map}}
\DeclareMathOperator{\cofiber}{\mathrm{cofiber}}

\DeclareMathOperator{\hofib}{\mathrm{hofib}}

\DeclareMathOperator{\Alt}{\mathrm{Alt}}

\DeclareMathOperator{\DM}{\mathrm{DM}}

\DeclareMathOperator{\Hom}{\mathrm{Hom}}
\DeclareMathOperator{\Isom}{\mathrm{Isom}}

\DeclareMathOperator*{\hocolim}{\mathrm{hocolim}}

\DeclareMathOperator{\Spec}{\mathrm{Spec}}
\DeclareMathOperator{\Spf}{\mathrm{Spf}}

\newcommand\SGpm{S\mathbb{G}_{\pm}}
\newcommand\Sbb{\mathbb{S}}

\newcommand{\BP}{\mathrm{BP}}
\newcommand{\MU}{\mathrm{MU}}
\newcommand{\MSU}{\mathrm{MSU}}
\newcommand{\MUP}{\mathrm{MUP}}
\newcommand\BSpin{\mathrm{BSpin}}
\newcommand{\MSpin}{\mathrm{MSpin}}
\newcommand\BString{\mathrm{BString}}
\newcommand{\tmf}{\mathrm{tmf}}
\newcommand{\BU}{\mathrm{BU}}
\newcommand{\BSU}{\mathrm{BSU}}
\newcommand{\BO}{\mathrm{BO}}
\newcommand{\BSp}{\mathrm{BSp}}
\newcommand{\TMF}{\mathrm{TMF}}
\newcommand\KU{\mathrm{KU}}
\newcommand\ku{\mathrm{ku}}
\newcommand\ko{\mathrm{ko}}
\newcommand{\BUsix}{\mathrm{BU} \langle 6 \rangle}
\newcommand{\MUsix}{\mathrm{MU} \langle 6 \rangle}

\newcommand\Gm{\hG_m}
\newcommand\Gmhat{\hG_m}
\newcommand\BGLS{\mathrm{BGL}_1(\mathbb{S})}
\newcommand\BGLE{\mathrm{BGL}_1(E)}

\newcommand\LMod{\mathrm{LMod}}
\newcommand\FG{\mathcal{FG}}
\newcommand\GL{\mathrm{GL}_1}
\newcommand\GLE{\mathrm{GL}_1(E)}
\newcommand\glS{\mathrm{gl}_1(\mathbb{S})}
\newcommand\glE{\mathrm{gl}_1(E)}

\newcommand\detGamma{\Lambda^h\Gamma}
\newcommand\detGammatilde{\widetilde{\Lambda^h\Gamma}}
\newcommand\E{\mathbb{E}}
\newcommand\Spectra{\mathop{\mathrm{Spectra}}}
\newcommand\EinftyAlg{\mathop{\Einfty\text{-}\mathrm{Alg}}}
\newcommand\EinftyAlgw{\mathop{\Einfty\text{-}\mathrm{Alg}^w}}

\newcommand\EstarAlg{\mathop{E_*\text{-}\mathrm{Alg}}}

\newcommand\mfrak{\mathfrak{m}}
\newcommand\EkG{E_{k,\Gamma}}
\newcommand\KkG{K_{k,\Gamma}}
\newcommand\kps{}
\def\kps[#1]{k\llbracket #1\rrbracket}
\def\L{\mathcal{L}}

\newcommand{\rmH}{\mathrm{H}}

\newcommand\Estarcts{E_*^{\wedge}}
\newcommand{\F}{\mathbb{F}}
\newcommand\Fp{\mathbb{F}_p}
\newcommand\Fph{\mathbb{F}_{p^h}}
\newcommand\Ftw{\mathbb{F}_2}
\newcommand{\Z}{\mathbb{Z}}

\newcommand{\B}{{\mathrm{B}}}
\newcommand{\hG}{{\widehat{\mathbb{G}}}}

\newcommand{\Gbb}{{\mathbb{G}}}
\newcommand{\hSGpm}{hS\Gbb^{\pm}}
\newcommand{\hC}{{\widehat{C}}}

\let\loops\Omega
\let\susp\Sigma
\let\sm\wedge
\newcommand\loopsinfty{\loops^{\infty}}

\newcommand\suspinftypl{\susp^{\infty}_+}
\newcommand\Sdet{S\langle\mathrm{det}\rangle}
\newcommand\Einfty{\mathbb{E}_{\infty}}
\newcommand\CP{\mathbb{CP}}
\newcommand\CPinfty{\mathbb{CP}^{\infty}}
\newcommand\Krm{\mathrm{K}}

\newcommand\KFph{\mathrm{K}(\mathbb{F}_p,h)}
\newcommand\KZhpo{\mathrm{K}(\mathbb{Z},h+1)}
\newcommand\KFptw{\mathrm{K}(\mathbb{F}_p,2)}
\newcommand\KZthree{\mathrm{K}(\mathbb{Z},3)}
\newcommand\hAlg{\mathrm{hAlg}}

\newcommand\Ocal{\mathcal{O}}
\newcommand\Ical{\mathcal{I}}

\makeatletter
\renewcommand\H{\mathrm{H}}
\def\HR@#1{\@ifnextchar^{\HR@sup{#1}}{\@ifnextchar_{\HR@sub{#1}}{\mathop{#1}}}}
\def\HR@s#1{\@ifnextchar^{\HR@sup{#1}}{\@ifnextchar_{\HR@s@sub{#1}}{\mathop{#1}}}}
\def\HR@sup#1^#2{\mathop{#1^{#2}}}
\def\HR@sub#1_#2{\mathop{#1_{#2}}}
\def\HR@s@sub#1_#2{\mathop{#1{}_{#2}}}
\protected\def\HFp{\HR@s{\mathrm{H}\Fp}}
\makeatother
\newcommand\HZ{\mathrm{H}\Z}

\newcommand{\OS}[2]{\vphantom{#1_{#2}}\underline{\smash{#1}}_{\,#2}}

\protected\def\BPn<#1>{\BP\mylangle#1\myrangle}
\def\BPh{\BPn<h>}
\def\OSBPh{\OS{\BPh}}
\def\BPph{\OS{\BPh}{2\nu(h)}}

\let\c\circ
\def\s{*}

\newbox\mylanglebox
\newbox\mylangleboxscript
\newbox\mylangleboxscriptscript

\newbox\myranglebox
\newbox\myrangleboxscript
\newbox\myrangleboxscriptscript

\newbox\smalllanglebox
\newbox\smallranglebox

\def\myangletikz#1#2{
\hskip-2.35pt\tikz[x=0.0022ex, y=0.0022ex,xscale=0.12,yscale=#1,xscale=0.9,yscale=0.9,xscale=#2]
 \fill
(0, 138)--(857, 1927)--(1711, 3715)--(2091, 3708)--(1319, 1917)--(552, 138)--(552, 138)--(1319, -1641)--(2091, -3432)--(1711, -3439)--(857, -1651)--(0, 138);
 }

\setbox\mylanglebox = \hbox{\myangletikz{0.12}{1}}
\setbox\myranglebox = \hbox{\myangletikz{0.12}{-1}}
\setbox\mylangleboxscript = \hbox{\myangletikz{0.08}{1}}
\setbox\myrangleboxscript = \hbox{\myangletikz{0.08}{-1}}
\setbox\mylangleboxscriptscript = \hbox{\myangletikz{0.06}{1}}
\setbox\myrangleboxscriptscript = \hbox{\myangletikz{0.06}{-1}}

\setbox\smalllanglebox = \hbox{\myangletikz{0.10}{1}}
\setbox\smallranglebox = \hbox{\myangletikz{0.10}{-1}}

\def\mylangle{\mathchoice
     {\lower 0.163ex\copy\mylanglebox}
     {\lower 0.163ex\copy\mylanglebox}
     {\hskip-0.0698ex\lower 0.0698ex\copy\mylangleboxscript\hskip0.0698ex}
     {\hskip-0.116ex\lower 0.0931ex\copy\mylangleboxscriptscript\hskip0.116ex}
}
\def\myrangle{\mathchoice
    {\lower 0.163ex\copy\myranglebox}
    {\lower 0.163ex\copy\myranglebox}
    {\hskip-0.0698ex\lower 0.0698ex\copy\myrangleboxscript\hskip0.0698ex}
    {\hskip-0.116ex\lower 0.0931ex\copy\myrangleboxscriptscript\hskip0.5pt}
}

\def\mysmalllangle{\lower 0.163ex\copy\smalllanglebox}
\def\mysmallrangle{\lower 0.163ex\copy\smallranglebox}

\title{Wilson Spaces, Snaith Constructions, and Elliptic Orientations}
\date{\today}

\author{Hood Chatham}
\address{Department of Mathematics, Massachusetts Institute of Technology, Cambridge, MA, USA}
\email{hood@mit.edu}

\author{Jeremy Hahn}
\address{Department of Mathematics, Massachusetts Institute of Technology, Cambridge, MA, USA}
\email{jhahn01@mit.edu}

\author{Allen Yuan}
\address{Department of Mathematics, Massachusetts Institute of Technology, Cambridge, MA, USA}
\email{alleny@mit.edu}

\begin{document}

\begin{abstract}
We construct a canonical family of even periodic $\mathbb{E}_{\infty}$-ring spectra, with exactly one member of the family for every prime $p$ and chromatic height $n$.
At height $1$ our construction is due to Snaith, who built complex $K$-theory from $\mathbb{CP}^{\infty}$.
At height $2$ we replace $\mathbb{CP}^{\infty}$ with a $p$-local retract of $\mathrm{BU} \langle 6 \rangle$, producing a new theory that orients elliptic, but not generic, height $2$ Morava $E$-theories.

In general our construction exhibits a kind of redshift, whereby $\mathrm{BP}\langle n-1 \rangle$ is used to produce a height $n$ theory.
A familiar sequence of Bocksteins, studied by Tamanoi, Ravenel, Wilson, and Yagita, relates the $K(n)$-localization of our height $n$ ring to work of Peterson and Westerland building $E_n^{hS\mathbb{G}^{\pm}}$ from $\mathrm{K}(\mathbb{Z},n+1)$.
\end{abstract}

\maketitle

\setcounter{tocdepth}{1}
\tableofcontents

\vbadness 5000


\newpage
\section{Introduction}
Integral cohomology, complex $K$-theory, and topological modular forms are three of the most studied cohomology theories in algebraic topology.  The corresponding spectra, denoted $\HZ$, $\KU$, and $\TMF$, share several desirable properties:
\begin{enumerate}
\item They are $\Einfty$-ring spectra, endowing the cohomology theories with both products and power operations.
\item They are canonical, meaning they are constructed without arbitrary choices.
\item They are integral, meaning they exist before completion at any prime.
\end{enumerate}
These theories see progressively deeper information along the chromatic filtration of the stable homotopy category, but they are increasingly difficult to construct.  Constructing similar theories at height $3$ and beyond is a highly technical area of algebraic topology, with deep connections to the theory of automorphic forms \cite{TAF}.

Here, we describe new cohomology theories represented by spectra satisfying (1) -- (3).
Our theories are straightforward to construct, and there is exactly one of them for each prime number $p$ and chromatic height $h \ge 1$.
Except for low values of $h$, our theories do not exist before $p$-localization, but they do exist before any sort of completion.

Our starting point is a theorem of Snaith \cite{Snaith} that constructs complex $K$-theory out of integral cohomology.  More precisely, Snaith built complex $K$-theory from $\CPinfty \simeq \Omega^{\infty} \Sigma^2 \HZ$:

\begin{cnstr}[Snaith]
Let $\beta\colon S^2=\CP^1 \to \CPinfty$ denote the generator of $\pi_2(\CPinfty)$.
Since $\CPinfty$ is an infinite loop space, its suspension spectrum $\suspinftypl \CPinfty$ is an $\Einfty$-ring spectrum.
Taking the mapping telescope along multiplication by $\beta$ produces another $\Einfty$-ring spectrum $\suspinftypl\CPinfty[\beta^{-1}]$ \cite[Section VIII]{EKMM}, and there is an equivalence
\[\suspinftypl \CPinfty[\beta^{-1}] \simeq \KU.\]
\end{cnstr}

We iterate this procedure, using $\KU$ to construct our new height $2$ theory at $p=2$.

\begin{cnstr}
Consider the space
\[\BUsix \simeq \Omega^{\infty} \tau_{\ge 6} \KU,\]
which is the $5$-connected cover of the space $\BU$.  The first nontrivial homotopy group of $\BUsix$ is $\pi_{6} \BUsix \cong \Z$.  Choose a generator $x_{6}\colon S^6 \to \BUsix$, and let $R$ denote the $\Einfty$-ring spectrum
\[R = \suspinftypl \BUsix [x_6^{-1}].\]
\end{cnstr}

We prove the following theorems:

\begin{thm}
The ring $R$ has torsion-free homotopy groups concentrated in even degrees.  In particular, $R$ is complex orientable.
\end{thm}

\begin{thm} \label{thm:BU6intro}
Let $K(n)$ denote a Morava $K$-theory of height $n$ at the prime $p$.  Then $K(n) \otimes R=0$ if and only if $p=2$ and $n \ge 3$.
\end{thm}


After localization at the prime $2$, both $R$ and $\TMF$ are chromatic height $2$ theories.  To relate them, recall that $\TMF$ is a homotopy limit of \emph{elliptic spectra} \cite{AHS,Lurie-ellipticII}.  Examples of particular interest to us are the \emph{elliptic Morava $E$-theories}, which we review in \Cref{sec:EllipticOrientation}.  There is an elliptic Morava $E$-theory $E_{k,\hC}$ associated to any supersingular elliptic curve $C$ over a perfect field $k$ of positive characteristic.

In \Cref{sec:Orientation} and \Cref{sec:EllipticOrientation}, we explore what it means to give a ring map from $R$ into an elliptic Morava $E$-theory.  We study this for various meanings of `ring map,' ranging from homotopy commutative to $\Einfty$.
The following theorem is an example of the sort of results we obtain:

\begin{thm} \label{thm:intro-elliptic}
Suppose that $E_{k,\hC}$ is an elliptic Morava $E$-theory associated to a supersingular elliptic curve $C$ over a perfect field $k$ of characteristic $2$.  Then there is a natural homotopy equivalence between
\begin{enumerate}
\item The space of $\Einfty$-ring maps $\MSU \to E_{k,\hC}$, and
\item The subspace of $\Einfty$-ring maps $R \to E_{k,\hC}$ that respect the Weil pairing on $C$.
\end{enumerate}
\end{thm}

The subspace in $(2)$ is the union of a collection of path components in the space of $\Einfty$-ring homomorphisms $R \to E_{k,\hC}$.  We clarify what it means to respect the Weil pairing in the more precisely stated \Cref{dfn:WeilSpace} and \Cref{thm:elliptic-orientations-main}.


\begin{rmk}
There is geometric interest in $\Einfty$-ring maps out of $\MSU$, since such maps represent highly structured invariants of manifolds up to bordism.
\end{rmk}

\begin{rmk}
Morava $E$-theories that are not elliptic often fail to receive even homotopy commutative ring maps from $R$ (cf. \Cref{exm:dieudonneexterior}).  In \Cref{cor:homotopy-orientations} we determine exactly which Morava $E$-theories receive such homotopy ring maps.
\end{rmk}

\subsection{The construction at a general prime and height}
\begin{cnv}
For the remainder of this paper, we fix a prime number $p$ and work in the $p$-local category.  That is to say, all spectra and simply connected spaces are implicitly $p$-localized.
\end{cnv}

\begin{cnv}
If $E$ is a spectrum and $n$ an integer, we sometimes use the notation $\OS{E}{n}$ to denote
\[\OS{E}{n} = \loopsinfty \susp^{n}E.\]
\end{cnv}
\begin{cnv}
For each integer $h \ge 0$, we let $\nu(h)$ denote the integer \[\nu(h)=\frac{p^{h+1}-1}{p-1}.\]
\end{cnv}

\begin{dfn}
For each nonnegative integer $h$, let $\BPn<h>$ denote the spectrum obtained from $\BP$ by quotienting out the Hazewinkel generators $v_{h+1},v_{h+2},\ldots \in \pi_*(\BP)$.  We define
\[W_{h} = \Omega^{\infty} \Sigma^{2 \nu(h)} \BPn<h> = \OS{\BPn<h>}{2\nu(h)}.\]
Steve Wilson proved \cite{Wilson-thesisII} that any choice of generators $v_{h+1},v_{h+2},\ldots$ yields the same space $W_{h}$ --- we choose the Hazewinkel generators solely for concreteness (cf. \Cref{rmk:Hazewinkel-choice}).
\end{dfn}

\begin{rmk}
Note that $W_0 \simeq \CPinfty$ at all primes $p$, while $W_1 \simeq \BUsix$ only for $p=2$. At odd primes $W_1$ is a retract of $\BUsix$.
\end{rmk}

Our ultimate generalization of \Cref{thm:BU6intro} runs as follows:

\begin{thm} \label{thm:intro-main}
For each integer $h \ge 0$, let $R_{h}$ denote the $\Einfty$-ring spectrum
\[R_{h} = \suspinftypl W_h[x^{-1}],\]
where $x$ is a generator of the bottom nonzero homotopy group $\pi_{2 \nu(h)}(W_h) \cong \Z_{(p)}$.  Then:
\begin{enumerate}
\item The homotopy groups of $R_h$ are free $\Z_{(p)}$-modules and vanish in odd degrees. Hence, $R_h$ is complex orientable (\Cref{thm:evenorientation} and \Cref{cor:FreeZp}).
\item The ring $R_h$ is Landweber exact (\Cref{thm:Landweber}).
\item The chromatic localization $L_{K(n)} R_h$ vanishes if and only if $n > h+1$ (\Cref{thm:height}).
\end{enumerate}
\end{thm}

The last point states that $R_{h-1}$ has \emph{chromatic height} $h$ for each integer $h \ge 1$.
By the main result of \cite{HinfBousfield}, it is equivalent to the pair of computations $K(h) \otimes R_{h-1} \ne 0$ and $K(h+1) \otimes R_{h-1} = 0$.

While we are far from understanding the top chromatic localization $L_{K(h)} R_{h-1}$ as an $\Einfty$-ring spectrum, we combine the Bousfield--Kuhn functor \cite{Kuhn2} with work of Hopkins--Hunton \cite{HopkinsHunton} to identify $L_{K(h)} R_{h-1}$ as a spectrum:

\begin{thm} \label{thm:AdditiveSplitting}
For $h \ge 1$ a positive integer, let $E(h)$ denote the Johnson--Wilson theory $\BPn<h>[v_{h}^{-1}]$.  There is an equivalence
\[L_{K(h)} R_{h-1} \simeq L_{K(h)} \left( \bigvee \Sigma^{2\ell} E(h) \right),\]
where the right hand side is the completion of a countable wedge of even suspensions of $E(h)$.  We stress that this splitting is additive, but not at all multiplicative.
\end{thm}

\subsection{Bocksteins and the top chromatic localization}
In recent years there have been other generalizations of Snaith's theorem, also designed to produce $\Einfty$-rings of chromatic heights $h>1$ \cite{Westerland,Lurie-ellipticII, Peterson}.
Unlike our work, these works proceed entirely in the $K(h)$-local category, and invert $K(h)$-local Picard elements.
Notably, Lurie proved that all Morava $E$-theories admit such a construction \cite[Construction 5.1.1, cf. Proposition 4.3.13]{Lurie-ellipticII}.

While our ring $R_{h-1}$ is integral, its $K(h)$-localization $L_{K(h)} R_{h-1}$ is closely connected to constructions of Peterson and Westerland \cite{Peterson,Westerland}.  We demonstrate this connection by giving a different, and intrinsically $K(h)$-local, construction of $L_{K(h)} R_{h-1}$.

\begin{dfn}
Concatenating the $p$-Bockstein, $v_1$-Bockstein, $v_2$-Bockstein, \ldots, and $v_{h-1}$-Bockstein yields a sequence of spectra
\begin{align*}
\Sigma^h \HFp \to \Sigma^{h+1} \HZ &\to \Sigma^{2 \nu(1)+h-2} \BPn<1> \to \Sigma^{2\nu(2)+h-3} \BPn<2> \to \cdots \to \Sigma^{2\nu(h-1)} \BPn<h-1>,
\end{align*}
which we call the \emph{Bockstein tower}.
We refer to the long composite
\[\Sigma^h \HFp \to \Sigma^{2\nu(h-1)} \BPn<h-1>\]
as the \emph{Tamanoi Bockstein}.
\end{dfn}

\begin{rmk}
The Tamanoi Bockstein is studied by Tamanoi in \cite{Tamanoi}, where it is used to define his $\BP$ fundamental class.  It features in the proof by Hopkins and Ravenel that suspension spectra are harmonic \cite{Suspension-spectra-are-harmonic}, as well as in several other papers such as \cite{WesterlandSati, Ravenel-Wilson-Yagita}.
\end{rmk}

\begin{rmk}
Applying the functor $\suspinftypl \Omega^{\infty}(\--)$ to the above sequence of Bocksteins yields a sequence of $\Einfty$-ring spectra
\[\suspinftypl \KFph \to \suspinftypl \KZhpo \to \suspinftypl \OS{\BPn<1>}{2\nu(1)+h-2}  \to \cdots \to \suspinftypl W_{h-1},\]
and we will also refer to this composite of $\Einfty$-ring maps as the Tamanoi Bockstein.
\end{rmk}

\begin{dfn} \label{intro-z}
In Section \ref{sec:Bockstein}, we follow \cite[Lemma 3.8]{Westerland} in considering a $K(h)$-local splitting
\[L_{K(h)} \suspinftypl \KFph \simeq L_{K(h)} \Sbb \vee Z \vee Z^{\otimes 2} \vee \cdots \vee Z^{\otimes p-1},\]
where $Z$ is a $\otimes$-invertible $K(h)$-local spectrum, representing an element of the $K(h)$-local Picard group.  The inclusion
\[z\colon Z \to L_{K(h)} \suspinftypl \KFph\]
defines an element $z$ in the Picard graded homotopy groups $\pi_{\star} L_{K(h)} \suspinftypl \KFph$. Composing $z$ with the Tamanoi Bockstein defines a class $z \in \pi_{\star} L_{K(h)} \suspinftypl W_{h-1}$.
\end{dfn}

Our alternative construction of $L_{K(h)} R_{h-1}$ is provided by the following theorem:

\begin{thm} \label{thm:introBockstein}
For $h \ge 1$ an integer, let $z \in \pi_{\star} L_{K(h)} \suspinftypl W_{h-1}$ denote the Picard graded homotopy class of \Cref{intro-z}.  Then there is an equivalence of $K(h)$-local $\Einfty$-ring spectra
\[L_{K(h)} R_{h-1} \simeq L_{K(h)} \suspinftypl W_{h-1} [z^{-1}].\]
\end{thm}

\begin{rmk}
We may also use the classical $p$-Bockstein, which gives a map
\[L_{K(h)} \suspinftypl \KFph \to L_{K(h)} \suspinftypl \KZhpo,\]
to define a ring
\[B(0)=L_{K(h)} \suspinftypl \KZhpo[z^{-1}].\]
The ring $B(0)$ was studied by Westerland \cite{Westerland} when $p>2$ and by Peterson \cite{Peterson} for $p \gg 0$.  Westerland proved an equivalence
\[B(0) \simeq E_h^{h\SGpm},\]
where $E_h$ is the Morava $E$-theory associated to the Honda formal group law over $\Fph$ and $\SGpm$ is a certain subgroup of the Morava stabilizer group.  He showed that $B(0)$ is ``half the $K(h)$-local sphere'' meaning there is a fiber sequence
\[L_{K(h)} \Sbb \to B(0) \to B(0).\]
We refer the reader to \cite{Westerland} for details. 
\end{rmk}

\begin{rmk} \label{rmk:intro-tower}
Inverting $z$ along the Bockstein tower yields an entire sequence of $\Einfty$-ring spectra.  For example, if $h=4$ there is a sequence
\[
\begin{tikzcd}
L_{K(4)} \suspinftypl \Krm(\Z,5)[z^{-1}] \arrow{r}
& L_{K(4)} \suspinftypl\OS{\BPn<1>}{2\nu(1)+2}[z^{-1}] \ar[out=south, in=north, looseness=0.4, overlay,ld] \\
L_{K(4)} \suspinftypl \OS{\BPn<2>}{2\nu(2)+1}[z^{-1}] \arrow{r}
& L_{K(4)}\OS{\BPn<3>}{2\nu(3)}[z^{-1}] \simeq L_{K(4)} R_3.
\end{tikzcd}
\]
Westerland \cite{Westerland} identified the first element of the sequence as $E_4^{h\SGpm}$.
\Cref{thm:AdditiveSplitting} relates the final element to a completed wedge of copies of $E_4$.
The middle elements of the sequence have not yet been studied.
\end{rmk}

\begin{qst}
Can the second element of the sequence,
\[L_{K(h)} \suspinftypl \OS{\BPn<1>}{2\nu(1)+h-2}[z^{-1}],\]
be identified in more familiar terms when $h>2$?  Is it in any way related to a fixed point of Morava $E$-theory by a subgroup of the Morava stabilizer group?
\end{qst}

\subsection{Further open questions}

We end this introduction by advertising a few additional open questions:

\begin{qst}
Since the ring $R_{h-1}$ is complex orientable, $\pi_*(R_{h-1})$ carries a formal group.  Can this ring and formal group be described algebro-geometrically?  Note that, since it is Landweber exact, $R_{h-1}$ is determined as a homotopy commutative ring by this formal group.  
\end{qst}

\begin{qst}[cf. \Cref{sec:Orientation}] \label{qst:IntroOrientation}
For any formal group $\Gamma$ of height $h$ over a perfect field $k$ of characteristic $p$, recall that there is an associated Morava $E$-theory $E_{k,\Gamma}$ \cite{GoerssHopkins}.  For what height $h$ Morava $E$-theories $E_{k,\Gamma}$ does there exist a homotopy commutative ring map
\[R_{h-1} \to E_{k,\Gamma}?\]
Which of these homotopy commutative ring maps admit $\Einfty$ lifts?
\end{qst}

\begin{qst}[cf. \Cref{qst:muorient} in \Cref{sec:Snaith}]
Are there $\Einfty$-ring homomorphisms $\MU \to R_{h-1}$?
\end{qst}

\begin{qst}[cf. \Cref{qst:BPfree} in \Cref{sec:InfiniteHeight}]
Let $x_4 \in \pi_4(\BSU)$ denote a generator.  Does the $\Einfty$-ring spectrum
\[\suspinftypl \BSU[x_4^{-1}]\]
admit an additive splitting into a wedge of suspensions of $\BP$?  In \Cref{rmk:BPfree} we observe that this spectrum is a retract of a wedge of suspensions of $\BP$.
\end{qst}

\begin{qst} In \cite{CSYCycl}, Carmeli, Schlank, and Yanovski observe that the sequence of $K(n)$-local rings in \Cref{rmk:intro-tower} admit canonical $T(n)$-local lifts, essentially because the Picard element $z \in \pi_{\star} L_{K(n)} \Sigma^{\infty}_+\mathrm{K}(\mathbb{F}_p,n)$ lifts to the $T(n)$-localization. Are the Bousfield classes of these $T(n)$-local rings distinct, and do these Bousfield classes redshift under algebraic $K$-theory?
\end{qst}

\subsection{Organization}

The paper begins with an explanation of how we arrived at our definition of the rings $R_{h}$.  In \Cref{sec:Wilson}, we review Steve Wilson's PhD thesis work \cite{Wilson-thesisII}, which defines the spaces $W_h$ and proves that they enjoy remarkable properties.  In \Cref{sec:Snaith}, we explain exactly how these properties ensure that $R_h$ has even-concentrated, torsion-free homotopy.  Additionally, we prove that $R_h$ is Landweber exact, and use this to prove \Cref{thm:AdditiveSplitting}.

In \Cref{sec:InfiniteHeight} we explore a few variations on our constructions that are of infinite chromatic height -- this is largely tangential to the rest of the paper.  We justify the claim that $\Sigma^{\infty}_+ \BUsix[x_6^{-1}]$ has infinite height when $p>2$.

In \Cref{sec:Bockstein} we give our alternate construction of $L_{K(h)} R_{h-1}$, using the Tamanoi Bockstein, and we also prove that this top chromatic localization is non-zero.

In \Cref{sec:Orientation} we determine exactly which height $2$ Morava $E$-theories receive homotopy commutative ring maps from $R_1$.  We also characterize, at any height, which Morava $E$-theories receive $\mathbb{E}_\infty$-ring homomorphisms from $B(0)$.

In \Cref{sec:EllipticOrientation} we recall the notion of an elliptic Morava $E$-theory, as well as work of Ando, Hopkins, Strickland, and Rezk on the $\sigma$-orientation.  We then use the $\sigma$-orientation to give our proof of \Cref{thm:intro-elliptic}.

Finally, \Cref{sec:UpperBound} proves that $L_{K(n)} R_{h-1}$ vanishes for $n>h$.  The proof uses the language of Hopf rings, and is largely independent of the rest of our work.

\subsection{Conventions}
We fix throughout this paper a prime $p$.
Unless we explicitly comment otherwise, all spectra and simply connected spaces will be implicitly $p$-localized.
For the convenience of the reader, we summarize here some of the conventions introduced later in the paper.

 We define
\[v_1,v_2,\ldots \in \pi_*(\BP) \cong \Z_{(p)}[v_1,v_2,\ldots]\]
to be the Hazewinkel generators \cite[A2.2.1]{greenbook}.
For each integer $h \ge 0$, we let $\BPn<h>$ denote the quotient spectrum $\BP/(v_{h+1},v_{h+2},\ldots)$, and we let $E(h)$ denote the Johnson--Wilson theory $\BPn<h>[v_{h}^{-1}]$.

If $E$ is a spectrum and $m$ an integer, we sometimes use the notation $\OS{E}{m}$ to denote
\[\OS{E}{m} = \Omega^{\infty} \Sigma^m E.\]
For any integer $m \ge 0$, we let $\nu(m)$ denote the integer
\[\nu(m) = \frac{p^{m+1}-1}{p-1}.\]

We use $W_{h,2k}$ to refer to the space
\[W_{h,2k} = \Omega^{\infty} \Sigma^{2k} \BPn<h>,\]
and use $W_h$ as shorthand for
\[W_h=W_{h,2\nu(h)}.\]

The ring $R_{h,2k}$ is obtained by inverting the bottom cell $S^{2k} \to \Sigma^{\infty}_+ W_{h,2k}$, and we denote $R_{h,2\nu(h)}$ simply by $R_{h}$.

The symbol $K(h)$ always denotes the homotopy associative, $(2p^h-2)$-periodic Morava $K$-theory associated to the Honda formal group law over $\Fp$.

We will generally use $\Gamma$ to denote a finite height $h\geq 1$ formal group over a perfect field $k$ of characteristic $p$, with $n$-torsion schemes $\Gamma[n]$.  We let $E_{k,\Gamma}$ denote the associated Morava $E$-theory, which is a $2$-periodic $\Einfty$-ring spectrum. We let $K_{k,\Gamma}$ denote the associated $2$-periodic Morava $K$-theory, which we view as a homotopy associative ring spectrum equipped with a homotopy ring map from $E_{k,\Gamma}$.

We fix a $K(h)$-local splitting \cite[Lemma 3.8]{Westerland}
\[L_{K(h)} \suspinftypl \KFph \simeq L_{K(h)}\Sbb \vee Z \vee Z^{\otimes 2} \vee \cdots \vee Z^{\otimes {p-1}},\]
and let $z \colon Z \to L_{K(h)} \suspinftypl \KFph$ denote the inclusion of the wedge summand.  We also fix an isomorphism of $K(h)_*$-algebras
\[K(h)_*(\KFph) \cong K(h)_*[f_0]/(f_0^p-(-1)^{h-1}v_hf_0),\]
such that $K(h)_*(z)$ has image the $K(h)_*$-submodule spanned by $f_0$.  The degree of $f_0$ is $2\nu(h-1)$.

\subsection{Acknowledgements}

We thank Craig Westerland for some very enlightening discussions, particularly regarding Sections \ref{sec:Bockstein} \& \ref{sec:Orientation}.
We also owe thanks to Eric Peterson, Piotr Pstr\k{a}gowski, Doug Ravenel, Andrew Senger, and Steve Wilson, and to the anonymous referee for numerous clarifying and enlightening remarks.
Most of all we thank Mike Hopkins, Jacob Lurie, and Haynes Miller, who served as the three authors' three PhD advisors and endured countless conversations about this work over the last several years.
Through the course of the work, the second author was supported by NSF Grant DMS-1803273, and the third author by NSF Grant DGE-1122374.


\section{Wilson spaces} \label{sec:Wilson}
\begin{dfn}
\label{dfn:Wilson-space}
A simply connected space $W$ is a \textit{Wilson space} if:
\begin{enumerate}
\item For each $k \ge 0$, $\rmH_{2k}(W;\Z_{(p)})$ is a finitely generated free $\Z_{(p)}$-module, and
\[\rmH_{2k+1}(W;\Z_{(p)}) \cong 0.\]
\item For each $k \ge 1$, $\pi_{2k}(W)$ is a finitely generated free $\Z_{(p)}$-module, and
\[\pi_{2k+1}(W)\cong 0. \]
\end{enumerate}
\end{dfn}

\begin{rmk}
Condition (1) of \Cref{dfn:Wilson-space} may be summarized by the statement that Wilson spaces \emph{have even cells}. This means that they admit $p$-local cell decompositions with no odd-dimensional cells and only finitely many cells in each even dimension.  A Wilson space is a simply connected space that simultaneously has even cells and even homotopy groups.
\end{rmk}

\begin{exm}
The spaces $\CPinfty$, $\BU$, $\BSU$, and $\BUsix$ are Wilson spaces for all primes $p$.  The spaces $\BSp$ and $\BO$ are Wilson spaces for primes $p>2$.  The product of any finite number of Wilson spaces is a Wilson space.
\end{exm}

As part of his PhD thesis work \cite{Wilson-thesisII}, Wilson provided a complete classification of all Wilson spaces.  We recall his results:

\begin{dfn}
\label{dfn:W2k}
Let $h,k \ge 0$ denote nonnegative integers with $1 \le k \le \nu(h) = \frac{p^{h+1}-1}{p-1}$.  Denote by $W_{h,2k}$ the simply connected space
\[W_{h,2k} = \Omega^{\infty} \Sigma^{2k} \BPn<h>.\]
The space $W_{h,2k}$ is said to be \emph{atomic} if $\nu(h-1) < k \le \nu(h)$.  When $k=\nu(h)$, we will sometimes abbreviate $W_{h,2k}$ as $W_{h}$.
\end{dfn}

\begin{thm}[Wilson] \label{thm:Wilson-classification}
For each $h \ge 0$ and $1 \le k \le \nu(h)$, the space
$W_{h,2k}$
is a Wilson space.  Every Wilson space is a product of atomic Wilson spaces, and no atomic Wilson space can be expressed as a nontrivial product of other spaces.
\end{thm}

\begin{rmk} \label{rmk:Hazewinkel-choice}
\Cref{dfn:W2k} depends on a choice of $\BPn<h>$ as a $p$-local spectrum.  In this paper, we define $\BPn<h>$ as the quotient of $\BP$ by the regular sequence of Hazewinkel generators $v_{h+1},v_{h+2},\ldots \in \pi_*(\BP)$.  While a different choice of generators might yield a different spectrum $\BPn<h>$, it is a consequence of Wilson's work that the spaces $W_{h,2k}$ \emph{do not} depend on the choice of $\BPn<h>$.  The spectrum $\BPn<h>$ is likely unique as well: recent work of Angeltveit and Lind \cite{AngeltveitLind} shows that all choices of $\BPn<h>$ are equivalent after $p$-completion, and conjecturally this is true prior to $p$-completion. For example, $\BPn<1>$ is unambiguous as a $p$-local spectrum by a theorem of Adams and Priddy \cite{AdamsPriddy}.
\end{rmk}

\begin{exm}
Some of the spaces $W_{h,2k}$ may be identified in more familiar terms:
\begin{itemize}
\item At all primes $p$, $W_0=W_{0,2} \simeq \CPinfty$. \\
\item At $p=2$, $W_{1,2} \simeq \BU$, $W_{1,4} \simeq \BSU$, and $W_{1,6}=W_1 \simeq \BUsix$.  While both $W_{1,4}$ and $W_{1,6}$ are atomic, $\BU$ is \emph{not} atomic.  Indeed,
\[\BU \simeq \CPinfty \times \BSU \simeq W_{0,2} \times W_{1,4}.\]
\item At $p=3$, $W_{1,4} \simeq \BSpin$ and $W_1=W_{1,8} \simeq \BString$.
\end{itemize}
\end{exm}
We note here two corollaries of \Cref{thm:Wilson-classification} that are not immediate from the definition of Wilson spaces:

\begin{cor}
Every Wilson space admits at least one infinite loop space structure.
\end{cor}

\begin{proof}
Each atomic Wilson space is canonically an infinite loop space, and every Wilson space is a product of atomic components.
\end{proof}

\begin{cor}
Suppose that $W$ is a Wilson space.  Then the identity path-component of $\Omega^2 W$ is also a Wilson space.
\end{cor}

\begin{proof}
Since every Wilson space $W$ is a product of atomic Wilson spaces, it suffices to check the corollary for atomic Wilson spaces. For atomic Wilson spaces, the statement follows from the relation $\Omega^2 W_{h,2k} \simeq W_{h,2k-2}$.
\end{proof}

\begin{rmk}
For $k < \nu(h)$ we have the relationship $W_{h,2k} = \Omega^2 W_{h,2k+2}$.  However, the spaces $W_{h}$ cannot be written as $\Omega^2 W$ for any other Wilson space $W$.
\end{rmk}

\begin{rmk} \label{rmk:different-multiplications}
There can be many different infinite loop space structures on a single Wilson space $W$.
For example, $\BU$ is usually equipped with the infinite loop space structure arising from the addition of virtual rank $0$ vector bundles.  However, $\BU$ also has an infinite loop space structure arising from the tensor product of virtual rank $1$ vector bundles.
This ``multiplicative'' infinite loop space structure is given by $\Omega^{\infty} (\Sigma^4 \ku \vee \Sigma^{2} \HZ)$, while the additive structure is given by $\Omega^{\infty} \Sigma^2 \ku$.

Work of Adams and Priddy \cite{AdamsPriddy} shows, at each prime $p$, that $W_{1}$ has a unique $p$-local infinite loop space structure. In forthcoming joint work with Andrew Senger, the middle author uses the Angeltveit--Lind theorem \cite{AngeltveitLind} to prove that the $p$-completion of any atomic Wilson space admits a unique $p$-complete infinite delooping.
\end{rmk}

\begin{cnv}
For each $h \ge 0$ and $1 \le k \le \nu(h)$, we will regard $W_{h,2k}$ as an infinite loop space via the formula 
\[W_{h,2k} \simeq \Omega^{\infty} \Sigma^{2k} \BPn<h>.\]
This infinite loop space structure in principle depends on our choice of the Hazewinkel generators in the definition of $\BPn<h>$.  As explained in Remark \ref{rmk:Hazewinkel-choice}, this choice is conjecturally immaterial and has been proven irrelevant after $p$-completion.
\end{cnv}


\section{Snaith constructions} \label{sec:Snaith}
In \cite{Snaith}, Snaith provided elegant constructions of the complex $K$-theory spectrum $\KU$ and the periodic complex bordism spectrum $\MUP$.

\begin{thm}[Snaith]
Let
\[\beta\colon S^2 =\CP^{1} \to \CPinfty \to \BU\]
denote the bottom cell. Then there are equivalences of homotopy commutative ring spectra
\begin{align*}
\suspinftypl \CPinfty[\beta^{-1}] &\simeq \KU, \text{ and}\\
\suspinftypl \BU[\beta^{-1}] &\simeq \MUP.
\end{align*}
\end{thm}

\begin{rmk}
The equivalence $\suspinftypl \CPinfty[\beta^{-1}] \simeq \KU$ is in fact one of $\Einfty$-ring spectra.  The $\Einfty$-ring structure on $\suspinftypl \BU[\beta^{-1}]$ was studied in detail by the second and third author in \cite{ExoticMultiplication}, where it is proved that the equivalence $\suspinftypl \BU[\beta^{-1}] \simeq \MUP$ is $\E_2$ but not $\E_5$.
\end{rmk}

This paper arose as an attempt to axiomatize the properties of $\CPinfty$ and $\BU$ that make Snaith's constructions have reasonable homotopy groups.  We are led to the following result:

\begin{thm} \label{thm:evenorientation}
Let $W$ denote a Wilson space equipped with a chosen double loop space structure, and suppose that $x \in \pi_{2k}(W)$ is a homotopy class.  Then the $\E_2$-ring spectrum
\[\suspinftypl W[x^{-1}]\]
is complex orientable, with torsion-free homotopy groups concentrated in even degrees.
\end{thm}

\begin{proof}
If $k=0$, then the image of $x$ in $\pi_* \Sigma^{\infty}_+ W$ is $1$ and the statement is trivially true.  Suppose that $k>0$.  Because the space $W$ has even homotopy groups, there are no obstructions to making the following diagram
\[
\begin{tikzcd}
S^{2k} \arrow{d} \arrow{r}{x} & W. \\
\Sigma^{2k-2} \CPinfty \arrow[dashed]{ur}
\end{tikzcd}
\]
After suspending, this gives a map of spectra
\[\Sigma^{\infty} \Sigma^{2k-2} \CPinfty \longrightarrow \suspinftypl W\]
and we may further compose to obtain a chain of maps
\[\Sigma^{\infty} \Sigma^{2k-2} \CPinfty \longrightarrow \suspinftypl W \longrightarrow \suspinftypl W[x^{-1}] \stackrel{x^{-1}}{\longrightarrow} \Sigma^{2k} \suspinftypl W[x^{-1}].\]
By construction, this gives a complex orientation of $\suspinftypl W[x^{-1}]$, and hence a homotopy commutative ring homomorphism
\[\MU \to \suspinftypl W[x^{-1}].\]
This implies in particular that $\suspinftypl W[x^{-1}]$ is a homotopy $\MU$-module.  Using the unit map $\Sbb \to \MU$, the sequence of maps
\[\Sbb \otimes \suspinftypl W[x^{-1}] \to \MU \otimes \suspinftypl W[x^{-1}] \to \suspinftypl W[x^{-1}]\]
exhibits $\suspinftypl W[x^{-1}]$ as a retract of $\MU \otimes \suspinftypl W [x^{-1}]$.  To finish, we need only show that
$\MU_*(W)[x^{-1}]$ is torsion-free and concentrated in even degrees.  The Atiyah-Hirzebruch spectral sequence
\[\H_*(W;\pi_*(\MU)) \implies \MU_*(W)\]
shows that $\MU_*(W)$ is torsion-free and even concentrated.  This property cannot change upon inverting a class.
\end{proof}

\begin{rmk}
We suspect that some version of \Cref{thm:evenorientation} holds assuming less than a double loop space structure on $W$.
A similar remark applies to \Cref{thm:Landweber} below.
It is technically convenient to posit a double loop space structure because there is a good general theory for inverting elements in $\E_2$-ring spectra \cite[Appendix A]{MayNil}.
In practice, all Wilson spaces we are interested in come equipped with infinite loop space structures.
\end{rmk}

\begin{thm}\label{thm:Landweber}
Let $W$ denote a Wilson space equipped with a double loop space structure, and suppose $x \in \pi_{2k}(W)$ is a homotopy class.  For any choice of homotopy commutative ring map
\[\BP \to \suspinftypl W[x^{-1}],\]
the homotopy groups $\pi_{*}(\suspinftypl W[x^{-1}])$ are a Landweber exact $\BP_*$-module.
\end{thm}

\begin{proof}
Let $A$ denote the ring spectrum $\suspinftypl W[x^{-1}]$.  The map of homotopy ring spectra $\BP \to A$ defines a sequence $p,v_1,v_2,\ldots \in \pi_*(A)$, and our goal is to prove that this is a regular sequence.   By Theorem \ref{thm:evenorientation}, $p$ is not a zero-divisor in $\pi_*(A)$.  Assume by induction that we have shown the sequence $p,v_1,\ldots,v_{\ell-1}$ to be regular for some positive integer $\ell$. Let $I_{\ell}$ denote the ideal $(p, v_1, \ldots ,v_{\ell-1}) \subset \pi_*(A)$ and suppose that $a \in \pi_*(A)$ satisfies the equation
\[v_{\ell} a \equiv 0  \pmod{I_{\ell}}.\]
We would like to show that $a$ itself is zero in $\pi_*(A)/I_{\ell}$.

We first claim that it suffices to show that $a=0$ in $\BP_*(A) / \eta_R(I_{\ell})$ where $\eta_R(I_{\ell})$ denotes the image of the ideal $I_{\ell}$ under the right unit map $\eta_R\colon  \Sbb\otimes A \to \BP \otimes A$.
To see this, note that since $A$ is a homotopy commutative $\BP$-algebra, we have a sequence $A \to \BP \otimes A \to A$ of homotopy commutative rings such that the composite is the identity.  As a result, we have the diagram of commutative rings
\begin{equation*}
\begin{tikzcd}
\pi_*(A) \arrow[r]\arrow[d] & \BP_*(A) \arrow[r]\arrow[d] & \pi_*(A)\arrow[d] \\
\pi_*(A)/I_{\ell} \arrow[r] & \BP_*(A)/\eta_R(I_{\ell}) \arrow[r] & \pi_*(A)/I_{\ell}
\end{tikzcd}
\end{equation*}
such that the horizontal composites are the identity.  It follows that the map $\pi_*(A)/I_{\ell} \to \BP_*(A)/\eta_R(I_{\ell})$ is an injection, so it suffices to show that the image of $a$ in $\BP_*(A)/\eta_R(I_{\ell})$ is zero.

We recall two facts about Landweber ideals:
\begin{enumerate}
\item Since $I_{\ell}$ is an invariant ideal, $\BP_*(A)/\eta_R(I_{\ell}) = \BP_*(A)/(p, v_1,\ldots, v_{\ell -1})$ where $v_i \in \BP_*(A)$ denotes the image of $v_i$ under the \emph{left} unit map $\BP \otimes \Sbb \to \BP \otimes A$.
\item There is a congruence
\[\eta_R(v_{\ell}) \equiv v_{\ell} \mod{(p,v_1,\ldots,v_{\ell -1})}\]
in $\BP_*(A)$ because the analogous equation holds in $\BP_* \BP$.
\end{enumerate}
Combining the above two facts, we see that
\[v_{\ell} a \equiv 0 \mod{(p,v_1,\ldots v_{\ell-1})}\]
in $\BP_*(A)$, and we would like to conclude that $a$ is itself congruent to zero modulo $(p,v_1,\ldots v_{\ell-1})$.  There is an equivalence of rings $\BP_*(A) = \BP_*(W)[x^{-1}]$, so for some $j$ the equation
\[x^j v_{\ell} a \equiv 0 \mod{(p,v_1,\ldots, v_{\ell -1})}\]
holds in $\BP_*(W)$.  The degeneration of the Atiyah-Hirzebruch spectral sequence for $\BP_*(W)$ implies that $\BP_*(W)$ is a free $\BP_*$-module, and so
\[x^j a\equiv 0 \mod{(p,v_1,\ldots, v_{\ell - 1})}\]
in $\BP_*(W)$.   Inverting $x$, we see that $a$ is zero in $\BP_*(A)/(p,v_1,\ldots, v_{\ell -1})$ as desired.
\end{proof}

A theorem of Hopkins and Hunton \cite{HopkinsHunton} restricts the homotopy type of any Landweber exact ring spectrum:

\begin{dfn}
A space $X$ is a \emph{weak product of Wilson spaces} if there is some sequence $A_1,A_2,\cdots$ of Wilson spaces such that
\[X \simeq \hocolim_{N\to\infty}\left( \prod_{i=1}^{N} A_i\right).\]
\end{dfn}

\begin{thm}[Hopkins--Hunton] \label{thm:Hopkins-Hunton}
Let $E$ denote a homotopy commutative, complex oriented, Landweber exact ring spectrum.  Suppose additionally that $\pi_*(E)$ is concentrated in even dimensions, and that each even homotopy group $\pi_{2n}(E)$ is free of countable rank over $\Z_{(p)}$.  Then, for each integer $\ell$, the connected component of the identity in
\[\Omega^{\infty} \Sigma^{2\ell} E\]
is a weak product of Wilson spaces.
\end{thm}

\begin{cor}
Let $W$ denote a Wilson space equipped with a double loop space structure, let $x \in \pi_{2*}(W)$ denote a homotopy class, and let $\ell$ denote an integer.  Then the connected component of the identity in
\[\Omega^{\infty} \Sigma^{2 \ell} \suspinftypl W[x^{-1}]\]
is a weak product of Wilson spaces.
\end{cor}

\begin{proof}
Use Theorem \ref{thm:Landweber} to apply Theorem \ref{thm:Hopkins-Hunton}.
\end{proof}

\begin{cor} \label{cor:FreeZp}
Let $W$ denote a Wilson space equipped with a double loop space structure, let $x \in \pi_{2*}(W)$ denote a homotopy class, and let $n$ denote any integer.  Then $\pi_{2n}(\suspinftypl W[x^{-1}])$ is a free $\mathbb{Z}_{(p)}$-module.
\end{cor}

\begin{proof}
The desired group is $\pi_2$ of the connected component of the identity in 
\[
\Omega^{\infty}\Sigma^{2-2n} \suspinftypl W[x^{-1}],
\]
so it suffices to show that $\pi_2$ of a weak product of Wilson spaces is a free $\mathbb{Z}_{(p)}$-module.  This is true because the homotopy groups of each Wilson space in the weak product are free $\mathbb{Z}_{(p)}$-modules, by definition.
\end{proof}

\begin{cor}
Let $W$ denote a Wilson space equipped with a double loop space structure, let $x \in \pi_{2*}(W)$ denote a homotopy class, and let $n$ denote a positive integer.  Then there is a $K(n)$-local splitting
\[L_{K(n)} \suspinftypl W[x^{-1}] \simeq L_{K(n)} \left( \bigvee \Sigma^{2\ell} E(n) \right)\]
of $L_{K(n)}\suspinftypl W[x^{-1}]$ into the completion of a countable wedge of even suspensions of the Johnson-Wilson theory $E(n)$.  We allow here for the possibility that this wedge be empty or finite.
\end{cor}

\begin{proof}
We may use the Bousfield--Kuhn functor $\Phi_n$ \cite{Kuhn} to calculate
\[L_{K(n)} \suspinftypl W[x^{-1}] \simeq \Phi_n(\Omega^{\infty} \suspinftypl W[x^{-1}] ).\]
It suffices to compute $\Phi_n$ of the connected component of the identity in
$\Omega^{\infty} \suspinftypl W[x^{-1}]$, since $\Phi_n(X)$ is equivalent to $\Phi_n(\tau_{\ge 1} X)$ for any space $X$.
The connected component of the identity is a weak product of Wilson spaces and the Bousfield--Kuhn functor commutes with both finite products and filtered colimits \cite[Proposition 3.21, Theorem 2.3]{Gijsvn}.  To finish, it thus suffices to calculate
$\Phi_n(A)$ whenever $A$ is an atomic Wilson space.  In particular, it suffices to note that
\[\Phi_n\left(\Omega^{\infty} \Sigma^{2r} \BPn<k> \right) \simeq \Sigma^{2r} L_{K(n)} \BPn<k> \simeq \begin{cases} 0 \text{ if } n>k, \\ L_{K(n)} \Sigma^{2r} E(k) \text{ if } n=k, \text{ and} \\ L_{K(n)}\left( \bigvee \Sigma^{2\ell} E(n) \right) \text{ if } n<k.  \end{cases}\]
The final bit of this calculation, computing $L_{K(n)} \BPn<k>$, may be found as \cite[Theorem 4.3]{HoveySadofsky}.
\end{proof}

The following are our primary examples of Snaith constructions:

\begin{dfn}\label{dfn:snaithwilson}
For each $h \ge 0$ and $1 \le k \le \nu(h)$, consider the Wilson space which is the infinite loop space
\[W_{h,2k} = \Omega^{\infty} \Sigma^{2k} \BPn<h>.\]
We denote by $R_{h,2k}$ the $\Einfty$-ring spectrum
\[R_{h,2k} = \suspinftypl W_{h,2k}[x_{2k}^{-1}],\]
where $x_{2k}$ is any generator of the bottom nonzero homotopy group of $W_{h,2k}$.
Just as we denote $W_{h,2\nu(h)}$ by $W_h$, we abbreviate
\[R_h = R_{h,2\nu(h)}.\]
\end{dfn}

\begin{qst}\label{qst:muorient}
Which of the rings $R_{h,2k}$ admit highly structured ring homomorphisms $MU \to R_{h,2k}?$
\end{qst}

\begin{exm}
By work of Snaith \cite{Snaith}, we may make the identifications
\begin{align*}
&R_0= R_{0,2} \simeq \KU\text{ as $\Einfty$-rings,}\\
\intertext{and at $p=2$}
&R_{1,2} \simeq \MUP \text{ as homotopy commutative rings.}\\
\intertext{The next simplest cases at $p=2$ are:}
&R_{1,4} \simeq \suspinftypl \BSU[x_4^{-1}] \text{, and}\\
&R_1=R_{1,6} \simeq \suspinftypl \BUsix [x_6^{-1}].
\end{align*}
\end{exm}

\begin{rmk} \label{rmk:notbottomBU}
Theorem \ref{thm:evenorientation} applies when inverting any element $x \in \pi_*(W)$.  In the rest of the paper, we only invert classes in the bottom nontrivial homotopy group.

To comment briefly on what happens in general, consider the case $W=\BU$.  Instead of inverting the Bott element $\beta \in \pi_2(\BU)$, one might instead invert a generator $x_6$ of $\pi_6(\BU)$.  There is an $\Einfty$-ring homomorphism
\[R_1=\suspinftypl \BUsix[x_6^{-1}] \longrightarrow \suspinftypl \BU[x_6^{-1}].\]
At $p=2$, we will prove that $R_1$ is $K(2)$-locally nontrivial but $K(3)$-acyclic.  Though we do not prove it here, computations show that $\suspinftypl \BU[x_6^{-1}]$ is $K(1)$-locally nontrivial but $K(2)$-acyclic.  In general, we expect that inverting classes other than the bottom cell yields less interesting ring spectra, but it would be good to know if the results of such inversions can be decomposed into smash products of simpler, well-known ring spectra.
\end{rmk}


\section{The infinite height case} \label{sec:InfiniteHeight}
The Snaith constructions $\KU \simeq \suspinftypl \CPinfty [\beta^{-1}]$ and $\MUP\simeq \suspinftypl \BU[\beta^{-1}]$ are both complex-oriented, but the former is supported at only finitely many chromatic heights.
We view this dichotomy as related to the fact that $\CPinfty$ is not the two-fold loop space of a Wilson space, while $\BU=\Omega^2 \BSU$ is.
\begin{thm} \label{thm:infinite-height}
Let $W$ and $W'$ be Wilson spaces such that $W = \Omega^2 W'$.
Suppose that $W$ is $2k$-connective and that $x \in \pi_{2k}(W)$ induces a surjection $H^{2k}(W;\Z_{(p)})\to H^{2k}(S^{2k};\Z_{(p)})$.
Then there is an $\E_2$-ring homomorphism
\[\suspinftypl W [x^{-1}] \to \MUP.\]
In particular, for no chromatic height $n$ is $L_{K(n)} \suspinftypl W [x^{-1}]=0$, since there are no ring homomorphisms from the zero ring to a nonzero ring.
\end{thm}

\begin{proof}
Note that the element $x \in \pi_{2k} (W)$ may also be viewed as a map $S^{2k+2} \to W'$. We begin by recalling that $\E_2$-ring maps
\[\suspinftypl W \longrightarrow \MUP\]
are equivalent to double loop maps
\[W \to \GL(\MUP).\]
These are in turn the same as maps of pointed spaces
\[W' \to \B^2\GL(\MUP).\]
Since $W'$ admits an even cell decomposition and
\[\pi_i(\B^2\GL(\MUP)) \cong \pi_{i-2}(\MUP) \text{ for $i>2$,}\]
the relative cohomology groups
\[H^i(W',S^{2k+2};\pi_{i+1} \B^2\GL(\MUP))\]
vanish.  Thus, there are no obstructions to forming the dashed arrow in the diagram
\[
\begin{tikzcd}
S^{2k+2} \arrow{d}{x} \arrow{r}{\beta^{2k}} \arrow{d} & \B^2\GL(\MUP) \\
W'. \arrow[dashed]{ur}
\end{tikzcd}
\]
Applying $\Omega^2$, we get a double loop map
\[W \longrightarrow \GL(\MUP),\]
and hence an $\E_2$-algebra map
\[\suspinftypl W \longrightarrow \MUP\]
that sends $x$ to $\beta^{2k}$.  Since $\beta^{2k}$ is invertible in $\pi_*(\MUP)$, we obtain the desired $\E_2$-ring homomorphism
\[\suspinftypl W[x^{-1}] \longrightarrow \MUP.\]
\end{proof}

\begin{cor}
Suppose $h \ge 0$ and $1 \le k < \nu(h)$.  For all integers $n \ge 0$,
\[L_{K(n)}R_{h,2k} \ne 0.\]
\end{cor}

\begin{exm}
At the prime $p=3$,
\[\BUsix = \Omega^{\infty} \Sigma^6 ku \simeq \Omega^{\infty} \left( \Sigma^6 \BPn<1>\vee \Sigma^8 \BPn<1> \right) = W_{1,6} \times W_{1,8}\]
as infinite loop spaces.  It follows that
\[ \suspinftypl \BUsix [x_6^{-1}] \simeq R_{1,6} \otimes \suspinftypl W_{1,8}.\]
We have just learned that $K(n)_*(R_{1,6}) \ne 0$ for all $n \ge 0$.  On the other hand, $K(n)_*(W_{1,8}) \ne 0$ because $W_{1,8}$ has even cells. It follows from the K\"unneth isomorphism that $K(n)_*(\suspinftypl \BUsix [x_6^{-1}]) \ne 0$.  Analogous arguments at larger primes justify the statement from the introduction that $\suspinftypl \BUsix[x_{6}^{-1}]$ has infinite chromatic height when $p>2$.
\end{exm}

\begin{qst} \label{qst:BPfree}
Suppose that $h \ge 0$ and $1 \le k < \nu(h)$.  Is the spectrum $R_{h,2k}$ a wedge of even suspensions of $\BP$?
\end{qst}

\begin{rmk} \label{rmk:BPfree}
The answer to \Cref{qst:BPfree} is yes for $p=2$, $h=1$, and $k=1$, by Snaith's computation
\[R_{1,2} \simeq \MUP.\]
The work of Ravenel and Wilson \cite{Ravenel-Wilson} shows that $\BP_*(\OS{\BPn<h>}{2k})$ is a polynomial ring over $\BP_*$ when $k < \nu(h)$. It follows in this case that $\BP_*(R_{h,2k})$ is a free $\BP_*$-module.  Since $R_{h,2k}$ is complex-oriented, it is a retract of $\BP \otimes R_{h,2k}$. This shows that $R_{h,2k}$ is a retract of a wedge of even suspensions of $\BP$.  If $R_{h,2k}$ were connective, it would follow that $R_{h,2k}$ were a wedge of copies of $\BP$, but we do not know how to conclude something like this in the nonconnective setting.
\end{rmk}

It remains to analyze the chromatic localizations of $R_{h}=R_{h, 2\nu(h)}$, where \Cref{thm:infinite-height} does not apply.  Most of the rest of this paper will be a meditation on this problem.

\begin{rmk}
Suppose $h \ge 0$.  Start with the quotient map of spectra
\[\BPn<h+1> \to \BPn<h>\]
and apply the functor $\suspinftypl \Omega^{\infty} \Sigma^{2\nu(h)} \left(\--\right)$.  After inverting the bottom class, one obtains a map of $\Einfty$-ring spectra
\[R_{h+1,2\nu(h)} \to R_{h},\]
where the source has infinite chromatic height.  In the case $h=0$ and $p=2$ this construction is familiar: it produces the map
\[\MUP \to \KU\]
featured in, e.g., \cite{GepnerSnaith}.
\end{rmk}


\section{The Bockstein tower} \label{sec:Bockstein}
Throughout this section we fix an integer $h \ge 1$.   In \Cref{sec:UpperBound}, we will prove that the ring $R_{h-1}=R_{h-1,2\nu(h-1)}$ has chromatic height $h$.  In this section, we study the top non-trivial chromatic localization $L_{K(h)} R_{h-1}$.

Our main theorem is an alternate construction of $L_{K(h)} R_{h-1}$.  As we will explain, the construction naturally produces a sequence of $K(h)$-local $\Einfty$-rings
\[B(0) \to B(1) \to \cdots \to B(h-1) = L_{K(h)} R_{h-1}.\]
We find this sequence extremely interesting, and, for the moment, mysterious.  The ring $B(0)$ was studied extensively by Westerland \cite{Westerland} for $p>2$, and by Peterson \cite{Peterson} for large primes $p \gg 0$.

\begin{rmk}
We freely use the language of \cite{HopkinsPic}, and in particular assume familiarity with the $K(h)$-local category and the notion of Picard graded homotopy.
\end{rmk}

\subsection{A \texorpdfstring{$K(h)$}{K(h)}-local idempotent}\label{sub:Z}

\begin{rec}
\label{Kh-homology-of-KFph}

Let $K(h)$ denote the $(2p^h-2)$-periodic Morava $K$-theory associated to the Honda formal group law over $\Fp$, with $\pi_* K(h) \cong \Fp[v_h^{\pm}]$ and $|v_h|=2p^h-2$.  Ravenel and Wilson computed \cite{KthyofEMspaces} that
\[K(h)_*(\KFph) \cong K(h)_*[f_0]/(f_0^p-(-1)^{h-1}v_hf_0),\]
with $|f_0|=2\nu(h-1)$.  Furthermore, $f_0$ is a coalgebra primitive.  See also \cite{HopkinsLurie} for the prime $p=2$.
\end{rec}

\begin{rmk}

Suppose $p=2$ and consider the natural splitting
\[\suspinftypl \Krm(\Ftw,h) \simeq \Sbb \vee \Sigma^{\infty} \Krm(\Ftw,h).\]
The above computation implies that $K(h)_*(\Sigma^{\infty} \Krm(\Ftw,h))\cong K(h)_*\{f_0\}$.  Thus, there is a $K(h)$-local splitting
\[L_{K(h)} \suspinftypl \Krm(\Ftw,h) \simeq L_{K(h)} \Sbb \vee Z,\]
where $Z=L_{K(h)} \Sigma^{\infty} \Krm(\Ftw,h)$ is an element of the $K(h)$-local Picard group.  Since $f_0^2=(-1)^{h-1} v_hf_0$, it must be the case that $Z^{\otimes 2} \simeq Z$, so in fact $Z \simeq L_{K(h)}\Sbb$ and there is a $K(h)$-local splitting
\[L_{K(h)} \suspinftypl \Krm(\Ftw,h) \simeq L_{K(h)} \Sbb \vee L_{K(h)} \Sbb.\]
\end{rmk}

\begin{rmk}\label{rmk:KFphsplitting}
As explained in \cite[Lemma 3.8]{Westerland}, for $p>2$ there is a $K(h)$-local splitting
\[L_{K(h)} \suspinftypl \KFph \simeq L_{K(h)} \Sbb \vee Z \vee Z^{\otimes 2} \vee \cdots \vee Z^{\otimes {p-1}}.\]
Here $Z$ is a $\otimes$-invertible $K(h)$-local spectrum, but $Z$ is not the sphere.  Instead, $Z$ represents an element of order $p-1$ in the Picard group of the $K(h)$-local category.  The image of the inclusion $K(h)_*(Z) \to K(h)_*(\KFph)$ is $K(h)_*\{f_0\}$, and for our purposes this may be taken as the defining property of $Z$.
\end{rmk}

\begin{dfn}
At all primes, we will use $z$ to refer to the inclusion
\[z\colon Z \to L_{K(h)} \suspinftypl \KFph,\]
so $z \in \pi_{\star} \left(L_{K(h)} \suspinftypl \KFph \right)$ is an element in Picard graded homotopy.
\end{dfn}

\begin{rmk}
Let $A$ denote a $K(h)$-local $\Einfty$-ring equipped with a map of $\Einfty$-rings
\[\suspinftypl \KFph \to A.\]
Then the composite
\[z\colon Z \to L_{K(h)} \suspinftypl \KFph \to A\]
determines an element $z$ in the Picard graded homotopy of $A$.  We may invert this element by taking the mapping telescope 
\[\colim(A \to Z^{-1} \otimes A \to Z^{-2}\otimes A \to \cdots) \]
of the $A$-module map 
\[z\colon Z \otimes A \to A,\]
and then $K(h)$-localizing the result, obtaining a $K(h)$-local $\Einfty$-ring $A[z^{-1}]$.  If the reader prefers not to contemplate inversion of Picard graded elements, it is equivalent to invert the class $z^{p-1} \in \pi_0(A)$.
\end{rmk}

\begin{prop} \label{prop:Idem}
Suppose that $A$ is a $K(h)$-local $\Einfty$-ring equipped with a map of $\Einfty$-rings
\[\phi\colon \suspinftypl \KFph \to A.\]
Then there is a splitting of $A$ such that the localization map
\[A \to A[z^{-1}]\]
is the projection onto a wedge summand.
\end{prop}

\begin{proof}
Let $K$ denote the $2$-periodic Morava $K$-theory associated to the Honda formal group law over $\mathbb{F}_p$, so $K$ is the $2$-periodification of $K(h)$.
We denote the periodicity generator in $\pi_2 K$ by $u$, so that $u^{p^h-1} = v_h$.
The natural map $K(h) \to K$ allows us to view elements of $K(h)_*(\KFph)$ as elements of $K_*(\KFph)$.

Now, let $w \in \pi_0(L_{K(h)}\Sigma^{\infty}_+\KFph)$ be the element $1-(-1)^{h-1}z^{p-1}$.   The map $\phi$ allows us to interpret $w$ and $z$ as elements in the Picard graded homotopy of $A$.  We claim that the natural map
\[A \to A[z^{-1}] \vee A[w^{-1}]\]
 is an equivalence.  Since the rings are $K(h)$-local, it suffices to check in $K$-homology.  Recall that $K_*(z)$ has image the $K_*$ multiples of an element $f_0\in K_*(\KFph)$, and that $f_0$ satisfies the equation
\[f_0^p = (-1)^{h-1}v_h f_0.\]
We define $f \in K_0(\KFph)$ to be 
\[f=u^{-\nu(h-1)} f_0.\]
We define $g \in K_0(\KFph)$ to be the Hurewicz image of $w$. 
Then one has the equations
\[g=1-(-1)^{h-1} f^{p-1} \hspace{3pt} \text{ and } \hspace{3pt} gf=0,\]
and one needs to show that the natural map
\[K_*(A) \to K_*(A)[f^{-1}] \oplus K_*(A)[g^{-1}]\]
is an equivalence.  But this is clear because inverting $f$ is the same as inverting the class $(-1)^{h-1}f^{p-1}$, and this is a complementary idempotent to $g$.

\end{proof}

\subsection{The rings \texorpdfstring{$B(i)$}{B(i)}}

\begin{cnstr}
The quotient maps
\[\BPn<n> \to \BPn<n-1>\]
have associated Bockstein maps
\[\BPn<n-1> \to \Sigma^{|v_n| + 1} \BPn<n>.\]
These assemble into a \emph{Bockstein tower}
\[\Sigma^h \HFp \to \Sigma^{h+1} \HZ \to \Sigma^{h+2p} \BPn<1>\to \cdots\to \Sigma^{2\nu(h-1)} \BPn<h-1>,\]
and we will refer to the long composite as the \emph{Tamanoi Bockstein}.
Applying $\Omega^{\infty}$ to the Bockstein tower gives a sequence of infinite loop maps:
\[\KFph \to \KZhpo \to \cdots \to W_{h-1}.\]
Suspending and $K(h)$-localizing then yields a sequence of $\Einfty$-ring spectra:
\[L_{K(h)} \suspinftypl \KFph \to L_{K(h)} \suspinftypl \KZhpo \to \cdots \to L_{K(h)} \suspinftypl W_{h-1}.\]
Finally, we invert $z$ to obtain a sequence of $\Einfty$-rings
\[L_{K(h)} \suspinftypl \KFph[z^{-1}] \to B(0) \to B(1) \to \cdots \to B(h-1).\]
In particular, we make the definition
\[B(i) = L_{K(h)} \suspinftypl\OS{\BPn<i>}{2\nu(i)+h-i-1}[z^{-1}].\]
\end{cnstr}

\begin{exm}
\label{exm:WesterlandB0}
Note that
\[B(0) = L_{K(h)} \suspinftypl \KZhpo [z^{-1}].\]
For odd primes $p$, Westerland \cite[Corollary 3.25]{Westerland} proved that
\[B(0) \simeq E_h^{\hSGpm},\]
where $E_h$ is the Morava $E$-theory associated to the Honda formal group law over $\Fph$.  Here, $S\Gbb^{\pm}$ denotes a specific subgroup of the Morava stabilizer group, defined by Westerland as the kernel of a twisted version of the determinant map (see \cite[Section 2.2]{Westerland} for details).
\end{exm}

\begin{rmk}
For $p>2$, Westerland defines a different $K(h)$-local Picard element $\Sdet$, and proves \cite[Theorem 1.2  \& Corollary 3.18]{Westerland} that $B(0)$ may alternatively be defined by inverting a class
\[\rho\colon \Sdet \to L_{K(h)} \suspinftypl \KZhpo.\]
This perspective is also taken up in \cite{Peterson} for $p \gg 0$.
\end{rmk}

Aside from the description of $B(0)$ in \Cref{exm:WesterlandB0}, the only other $B(i)$ that we consider partially understood is
\[B(h-1) = L_{K(h)} \suspinftypl W_{h-1}[z^{-1}].\]
The main theorem of this section is as follows:

\begin{thm} \label{thm:BocksteinMain}
There is an equivalence of $\Einfty$-ring spectra
\[B(h-1) \simeq L_{K(h)} R_{h-1}.\]
In other words, starting from $L_{K(h)} \suspinftypl W_{h-1}$, the same ring is obtained by inverting either the bottom cell $x_{2\nu(h-1)}\colon S^{2\nu(h-1)} \to W_{h-1}$ or the $K(h)$-local Picard class $z$.

\end{thm}

\begin{proof}
Throughout this proof, let $k:= \nu(h-1)$ for ease of notation.
There is a natural zig-zag of $\Einfty$-rings
\[L_{K(h)} \suspinftypl W_{h-1}[x_{2k}^{-1}] \rightarrow L_{K(h)}\suspinftypl W_{h-1} [x_{2k}^{-1},z^{-1}] \leftarrow L_{K(h)}\suspinftypl W_{h-1} [z^{-1}].\]
To show that both maps in this zig-zag are equivalences, it suffices to check that the Hurewicz image of $x_{2k}$ in $K(h)_*(W_{h-1})$ is in the image of $K(h)_*(z)$.  By \cite[Proposition 3.7]{Westerland} (and the analogous statement at $p=2$), the image of $z$ in $K(h)_*(\KFph)$ is generated, as a $K(h)_*$-module, by the element $f_0 = a_{(0)}\circ a_{(1)} \circ \cdots \circ a_{(h-1)}$, where $a_{(i)}\in K(n)_{2p^i}(\Krm(\Fp, 1))$.  Here, we use the Hopf ring notation of \cite[\S 8.3.1]{Ravenel-Wilson-Yagita}, and the $a_i$ are elements represented in the Atiyah-Hirzebruch spectral sequence
\[ \mathrm{H}_*(\Krm(\Fp, 1);K(h)_*) \implies K(h)_{*}(\Krm(\Fp, 1))\]
by the classes used to define $\tau_i$ in the dual Steenrod algebra $\mathcal{A}_*$.  It suffices to show that the Hurewicz image of $x_{2k}$ is a nonzero scalar multiple of $f_0$.

Consider the diagram
\begin{equation*}
\begin{tikzcd}
\suspinftypl \KFph \arrow[r, "\varphi_1"]\arrow[d, "\sigma_1"] & \suspinftypl W_{h-1} \arrow[d, "\sigma_2"] \\
\Sigma^h \HFp \arrow[r, "\varphi_2"] & \Sigma^{2k}\BPn<h-1>
\end{tikzcd}
\end{equation*}
induced by the natural transformation $\suspinftypl\Omega^{\infty} \to \mathrm{id}.$  Here, the horizontal arrows are induced by the Tamanoi Bockstein.  By the above remarks, there is an element $\tilde{f}_0\in H_{2k}(\KFph; \Fp)$ which is an Atiyah-Hirzebruch representative for $f_0$ and such that $\sigma_1(\tilde{f}_0)=\tau_0 \tau_1 \cdots \tau_{h-1} \in \mathcal{A}_*$.
The image $\varphi_2(\tau_0 \tau_1 \cdots \tau_{h-1})$ in
\[H_{2k} (\Sigma^{2k}\BPn<h-1>, \Fp )\cong \Fp\]
 is nonzero by \cite[Proposition 3.3]{Tamanoi}.  Since $H_{2k}(\sigma_2)$ is an isomorphism, it follows that $\varphi_1(\tilde{f}_0)$ is a nonzero element of $H_{2k}(W_{h-1},\Fp)\cong \Fp$.

Denote by $M$ the finite complex which is the union of the $(2k-1)$-skeleton of $\suspinftypl \KFph$ and the single $(2k)$-cell defined by $\tilde{f}_0$.
By the above discussion, since $W_{h-1}$ is $(2k-1)$-connected, $M$ fits into a diagram
\begin{equation*}
\begin{tikzcd}
M \arrow[r] \arrow[d,hook] &S^{2k}\arrow[d,"x_{2k}"]\\
\suspinftypl \KFph \arrow[r] & \suspinftypl W_{h-1},
\end{tikzcd}
\end{equation*}
where the top arrow is induced by collapsing onto the top cell and the map $S^{2k} \to \suspinftypl W_{h-1}$ is a choice of bottom cell $x_{2k}$.
To finish, consider the Atiyah-Hirzebruch spectral sequence computing the $K(h)$-homology of $\KFph$.  Recall from \Cref{Kh-homology-of-KFph} that
\[K(h)_*(\KFph) = K(h)_*[f_0]/(f_0^p = (-1)^{n-1}v_nf_0),\]
 and so the element $\tilde{f}_0$ is a permanent cycle representing $f_0$ with no classes in lower filtration.  It follows that the image of $f_0$ is the same as the image of $x_{2k}$ in $K(h)_*(W_{h-1})$.
\end{proof}

\begin{cor}
The natural map
\[L_{K(h)} \suspinftypl W_{h-1} \to L_{K(h)} R_{h-1}\]
is given by projection onto a wedge summand.
\end{cor}
\begin{proof}
This follows immediately from the combination of \Cref{thm:BocksteinMain} and \Cref{prop:Idem}.
\end{proof}

\begin{cor}\label{cor:lowerbound}
The localization $L_{K(h)} R_{h-1}$ is nonzero.
\end{cor}

\begin{proof}
It suffices to show that $K(h)_*(R_{h-1})$ is nontrivial.
We can compute this as $K(h)_*(W_{h-1})[x^{-1}],$ where $x$ is the Hurewicz image of a generator of the bottom homotopy group $\pi_{2\nu(h-1)} W_{h-1}$, and so it suffices to show that $x$ is not nilpotent in $K(h)_*(W_{h-1})$.

According to the proof of Theorem \ref{thm:BocksteinMain}, $x$ may alternatively be described as the image of $f_0$ under the Tamanoi Bockstein $K(h)_* \KFph \to K(h)_*(W_{h-1})$.  Since $f_0^p = (-1)^{h-1} v_h f_0$, it follows that $x^p = (-1)^{h-1} v_h x$.  Since $x$ is nonzero (as it is the bottom cell and the Atiyah--Hirzebruch spectral sequence for $K(h)_*(W_{h-1})$ collapses by evenness),  we see that it is not nilpotent.
\end{proof} 

\section{Orienting Lubin--Tate theories} \label{sec:Orientation}
If $\Gamma$ is a height $h$ formal group over a perfect field $k$ of characteristic $p$, the Goerss--Hopkins--Miller theorem \cite{GoerssHopkins} allows us to form an associated $\Einfty$-ring, the Morava $E$-theory $E=\EkG$.
In this section, we will be concerned with the following question:

\begin{qst} \label{qst:orientation}
For which height $h$ formal groups $\Gamma$ over perfect fields $k$ are there structured ring homomorphisms
\[R_{h-1} \to \EkG?\]
\end{qst}

Our interest in this question is sparked by the very elegant answer when $h=1$:

\begin{exm}
Recall that Snaith proved $R_0 \simeq \KU$, with $\Einfty$-ring structure given by the tensor product of vector bundles.  After $p$-completion, $\KU$ is a Morava $E$-theory.
Specifically, $\KU_p^{\wedge}$ is the Morava $E$-theory associated to the formal multiplicative group $\Gmhat$ over the field $\Fp$.
On the other hand, there are many height $1$ formal groups over $\Fp$ that are not isomorphic to $\Gmhat$ (see \Cref{exm:ht1dieudonne}), and their associated $E$-theories are less clearly tied to geometry.
We will soon see that homotopy commutative ring maps
\[R_0 \to \EkG\]
induce isomorphisms $\Gamma \cong \Gm$ of formal groups over $k$.
In the case that $\Gamma \cong \Gm$, the Goerss--Hopkins--Miller theorem provides an $\Einfty$-ring homomorphism $R_0 \to \EkG$.
Together these statements give a complete answer to the height $1$ case of \Cref{qst:orientation}.
\end{exm}

At a general height, note that any ring homomorphism
\[R_{h-1} \to \EkG\]
factors through $L_{K(h)} R_{h-1}$, because height $h$ Morava $E$-theories are $K(h)$-local.  Thus we may study such ring homomorphisms via the Bockstein tower
\[L_{K(h)}\suspinftypl \KZhpo[z^{-1}] = B(0) \to B(1) \to \cdots \to B(h-1)= L_{K(h)} R_{h-1}.\]

A homotopy ring map $R_{h-1} \to \EkG$ restricts to a homotopy ring map
\[ B(0) = L_{K(h)}\suspinftypl \KZhpo[z^{-1}]  \to \EkG.\]
In \Cref{sub:dieudonne}, we show that this imposes a strong constraint on the formal group $\Gamma$ -- namely, that it must have \emph{top exterior power} isomorphic to $\Gm$ in an appropriate sense (\Cref{Westerland5.5-strengthened}).  In \Cref{subsec:structuredB0}, we compute the space of $\E_{\infty}$ $B(0)$-orientations of any Morava $E$-theory; if $\Gamma$ has top exterior power $\Gm$, then we prove it receives an $\E_{\infty}$ $B(0)$-orientation.  Finally, in \Cref{sub:r1}, we discuss the problem of lifting $B(0)$-orientations to $R_1$-orientations at height $2$. In particular, we show that a height $2$ Morava $E$-theory $E_{k, \Gamma}$ receives a homotopy commutative $R_1$-orientation if and only the top exterior power of $\Gamma$ is isomorphic to $\Gm$.

\subsection{Exterior powers of formal groups and \texorpdfstring{$B(0)$}{B(0)}-orientations}\label{sub:dieudonne}

\begin{dfn}\label{dfn:moravaek}
Given a height $h$ formal group $\Gamma$ over a perfect field $k$ of characteristic $p$, we denote by $\KkG$ the associated $2$-periodic Morava $K$-theory.  This is a homotopy associative ring spectrum equipped with a homotopy ring map
\[q\colon \EkG \to \KkG.\]
In general, $\KkG$ is not homotopy commutative.
\end{dfn}

\begin{rmk}
In \cite{KthyofEMspaces}, Ravenel and Wilson computed that there is a non-canonical isomorphism
\[\KkG^0 ( \KZhpo ) \cong \kps[x]\] (at least for $\Gamma$ being the Honda formal group, with the general case being \cite[Theorem 2.4.10]{HopkinsLurie}).
The multiplication map
\[\KZhpo \times \KZhpo \to \KZhpo\]
then makes $\Spf \left( \KkG^0 \KZhpo \right)$ into a $1$-dimensional formal group over $k$.
Inspired by Ravenel--Wilson, Buchstaber--Lazarev \cite{BuchLaz}, followed by Hopkins--Lurie \cite{HopkinsLurie}, showed that the formal group $\Spf \left( \KkG^0 \KZhpo \right)$ is the \emph{top exterior power} of the formal group $\Gamma$ in an appropriate sense.  To explain this, we begin by recalling the theory of (covariant) Dieudonn\'{e} modules.
\end{rmk}

\begin{dfn}
Let $D_k$ denote the free associative algebra over the Witt vectors $W(k)$ on two operators $F$ and $V$ subject to the relations
\begin{align*}
FV=VF&=p&Fa &= \varphi (a) F&   V\varphi(a) &= aV,
\end{align*}
for $a\in W(k)$, where $\varphi\colon W(k)\to W(k)$ denotes the Witt vector Frobenius.  We will refer to the ring $D_k$ as the \emph{Dieudonn\'{e} ring} and left modules over $D_k$, denoted by $\mathrm{LMod}_{D_k}$ as \emph{Dieudonn\'{e} modules}.
\end{dfn}

We will be interested in Dieudonn\'{e} modules because of their relationship to formal groups through what is called the (covariant) \emph{Dieudonn\'{e} correspondence}.  This classical correspondence takes many forms \cite{Demazure}, but we state here a form suited to our uses:


\begin{thm}[{\cite[Theorem 12.5]{BuchLaz}}]
\label{thm:dieudonne}
Let $\FG_k$ denote the category of smooth $1$-dimensional commutative formal groups over a perfect field $k$ of characteristic $p$. Then there is a fully faithful functor
\[\DM\colon  \FG_k \to \LMod_{D_k}\]
such that if $\Gamma \in \FG_k$, then $\DM(\Gamma)$ is a finitely generated free $W(k)$-module of rank equal to the height of $\Gamma$.  
\end{thm}


\begin{rmk}
We point out to the reader that the conventions regarding Dieudonn\'{e} modules of formal groups differ slightly in the two references \cite{BuchLaz, HopkinsLurie} that we draw from, but they agree for the $p^t$-torsion subgroups.  
We take the convention of \cite{BuchLaz} that our Dieudonn\'{e} modules are free modules over $W(k)$ given by the inverse limit of the Dieudonn\'{e} modules of the $p^t$-torsion subgroups.  
\end{rmk}

\begin{exm}\label{exm:ht1dieudonne}
The Dieudonn\'{e} module $M$ associated to a height $1$ formal group $\Gamma$ over $\F_p$ is a free $\Z_p$-module of rank $1$.  Choosing a generator $u\in M$, one deduces that the Frobenius and Verschiebung operators on $M$ are given by the formulas
\begin{align*}
 Fu &= a^{-1} u &   Vu &= apu
\end{align*}
for some $a\in \Z_p^{\times}.$  Since the Witt vector Frobenius $\varphi$ is trivial, the constant $a$ is independent of the choice of generator $u$.  Moreover, it can be shown that $a$ is equal to $1$ exactly when $\Gamma$ is isomorphic to the multiplicative formal group \cite[Lemma 9.8 and preceding]{BuchLaz}.  
\end{exm}

\begin{cnstr}[Exterior powers of Dieudonn\'{e} modules {\cite[Definition 9.5]{BuchLaz}}]\label{cnstr:exterior}
Given a Dieudonn\'{e} module $M$, one can endow the exterior algebra $\Lambda^m_{W(k)}(M)$ with the structure of a Dieudonn\'{e} module by the formulas (cf. \cite[Theorem 5.4 and preceding]{BuchLaz})
\begin{align*}
V(x_1 \wedge \cdots \wedge x_m) &= Vx_1 \wedge \cdots \wedge Vx_m \\
F(Vx_1 \wedge \cdots \wedge Vx_{i-1} \wedge x_i \wedge Vx_{i+1} \wedge \cdots \wedge Vx_m) &= x_1 \wedge \cdots \wedge x_{i-1} \wedge Fx_i \wedge x_{i+1} \wedge \cdots \wedge x_m.
\end{align*}
\end{cnstr}


\begin{cnstr}[Top exterior power of formal groups {\cite[Corollary 3.5.5]{HopkinsLurie}}]
Given a formal group $\Gamma \in \FG_k$ of height $h$, Hopkins-Lurie construct (via its associated $p$-divisible group) a $p$-divisible group of height $1$ and dimension $1$ whose associated formal group we denote by $\Lambda^h \Gamma$.  

This construction is related to the above one by the formula 
\[\DM(\Lambda^h \Gamma) \cong \Lambda^h_{W(k)} \DM(\Gamma).\]
We refer to these as the \emph{top exterior powers} of the formal group and corresponding Dieudonn\'{e} module.  

\end{cnstr}

For us, the relevance of these exterior powers is the following theorem.  This theorem is closely inspired by work of Ravenel--Wilson \cite{KthyofEMspaces}, and was first expressed in the setting of Dieudonn\'{e} modules by Buchstaber--Lazarev \cite{BuchLaz}:

\begin{thm}[{Hopkins--Lurie}]
\label{KZh+1=det(G)}
If $\Gamma$ is a height $h$ formal group over a perfect field $k$ of characteristic $p$, then there is an isomorphism of formal groups
\[\Spf \left( \KkG^0 \KZhpo \right) \cong \Lambda^h \Gamma.\]
\end{thm}
\begin{proof}
    This follows from applying Cartier duality to \cite[Corollary 3.3.3]{HopkinsLurie} for $d=h$ and taking the colimit, noting that the Cartier dual of the left-hand side of their theorem is $\Spec K(n)^0(X)$ and the Cartier dual of $\Alt^{(h)}_{\mathbb{G}_0[p^t]}$ is what we denote $\Lambda^h\Gamma[p^t]$.  
\end{proof}

We now use these constructions to study $B(0)$ orientations of Morava $K$-theories.




\begin{dfn}
Let $\KkG$ denote a Morava $K$-theory associated to a height $h$ formal group $\Gamma$ over a perfect field $k$ of characteristic $p$.  We call a homotopy ring map
\[\suspinftypl \KZhpo \to \KkG\]
\emph{nontrivial} if it does not factor through the augmentation map $\suspinftypl \KZhpo \to \Sbb.$
\end{dfn}

The following theorem and its proof are heavily influenced by the results of Westerland \cite[Section 5.1]{Westerland}:

\begin{thm}
\label{Westerland5.5-strengthened}
Let $\KkG$ denote a Morava $K$-theory corresponding to a height $h$ formal group $\Gamma$ over a perfect field $k$ of characteristic $p$.
If there exists at least one non-trivial homotopy ring map $\suspinftypl \KZhpo \to \KkG$, then the top exterior power $\Lambda^{h} \Gamma$ of the formal group $\Gamma$ is isomorphic to $\Gm$.
\end{thm}

\begin{proof}
Recall that one can choose a (noncanonical) isomorphism 
\[\KkG^0(\KZhpo) \cong \kps[x].\]
Under this identification, the nontrivial homotopy ring map $\suspinftypl \KZhpo \to \KkG$
corresponds to some power series $1+f(x)$, where $f(x) \ne 0$ is divisible by $x$.
Consider the homotopy commutative diagram
\[
\begin{tikzcd}[column sep=huge]
\left(\suspinftypl \KZhpo\right)^{\otimes 2} \arrow{rr}{(1+f(x))^{\otimes 2}} \arrow{d}{m} && (\KkG)^{\otimes 2} \arrow{d}{m} \\
\suspinftypl \KZhpo \arrow{rr}{1+f(x)} && \KkG.
\end{tikzcd}
\]
The homotopy commutativity witnesses an equality between two classes in
\[\KkG^{0}\left(\suspinftypl \KZhpo^{\otimes 2}\right) \cong \kps[x_1,x_2].\]
Namely, we have the equation
\[1+f(x_1+_{\Lambda^h \Gamma} x_2) = (1+f(x_1))(1+f(x_2)).\]
Thus, $f$ provides a nonzero formal group homomorphism from $\Lambda^{h} \Gamma$ to $\Gm$.  By the following \Cref{maps-to-gm-is-gm}, we deduce that $\Lambda^{h} \Gamma\cong \Gm$.
\end{proof}

\begin{lem}
\label{maps-to-gm-is-gm}
Suppose that $k$ is a perfect field of characteristic $p$ and that $\Gamma$ is a formal group over $k$ with a nonzero homomorphism $f\colon \Gamma\to \Gm$. Then $\Gamma\cong \Gm$.
\end{lem}

\begin{proof}
Let $\phi\colon \Gamma\to \Gamma^{(1)}$ denote the relative Frobenius on $\Gamma$. 
In coordinates, $\phi(x)=x^p$ and the group law on $\Gamma^{(1)}$ satisfies $(x+_{\Gamma}y)^p = x^p+_{\Gamma^{(1)}}y^p$.  

Note that any homomorphism of formal groups has derivative which is either everywhere nonzero (in which case the homomorphism is an isomorphism) or identically zero.  In this latter case, the homomorphism factors as the relative Frobenius followed by another homomorphism of formal groups (cf. for instance \cite[III.3.3]{TMFbook}).  It follows that since $f$ is nonzero, there exists an integer $n\geq 0$ such that the map $f$ factors as $g\circ \phi^n$ for some isomorphism of formal groups $g\colon \Gamma^{(n)}\to \Gm$.  
Since $k$ is perfect and $\Gm^{(n)} \cong \Gm$,  we deduce that $\Gamma\cong \Gm$.
\end{proof}


If $K_{k,\Gamma}$ receives a homotopy ring map from $B(0)$, then precomposition with the natural maps 
\[\Sigma^{\infty}_+ \KZhpo \to L_{K(h)} \Sigma^{\infty}_+ \KZhpo \to B(0)\]
produces a non-trivial ring map $\Sigma^{\infty}_+ \KZhpo\to K_{k,\Gamma}$, which by \Cref{Westerland5.5-strengthened} means that $\Lambda^h \Gamma \cong \Gm$.
One consequence of this is that there are examples of height $2$ Morava $K$-theories over $\Fp$ that do not receive an orientation from $B(0)$.  

\begin{exm}\label{exm:dieudonneexterior}
Let $\Gamma$ be the Honda formal group over $\F_p$ at height $h=2$.  By \cite[Lemma 9.8 and preceding]{BuchLaz}, the Dieudonn\'{e} module of $\Gamma$ is given by $D_k/(V - F)$, which is easily seen to be a free $\Z_p$ module on two generators $w_1$ and $w_2$ subject to:
\begin{align*}
Fw_1 &= w_2 & Vw_1 &=w_2 \\
Fw_2 &= pw_1 & Vw_2 &= pw_1.
\end{align*}
By \Cref{cnstr:exterior}, $\DM(\Lambda^2 \Gamma) = \Lambda^2 \DM(\Gamma)$ has a single generator $w_1\wedge w_2$ with 
\[
V(w_1\wedge w_2) = Vw_1 \wedge Vw_2 = -p(w_1\wedge w_2).
\]
It follows from \Cref{exm:ht1dieudonne} that this is not isomorphic to the Dieudonn\'{e} module for the multiplicative group.  Thus, $\Lambda^2 \Gamma$ is not isomorphic to $\Gm$ and so $K_{\F_p, \Gamma}$ does not receive a homotopy ring map from $B(0)$. 
\end{exm}

\subsection{Structured \texorpdfstring{$B(0)$}{B(0)}-orientations of Morava \texorpdfstring{$E$}{E}-theories} \hfill \vspace{2ex} \\
\label{subsec:structuredB0}
By \Cref{Westerland5.5-strengthened}, nontrivial homotopy ring maps $\suspinftypl \Krm(\Z,n+1) \to K_{k,\Gamma}$ can exist only if $\Lambda^h \Gamma \cong \Gm$.  In particular, homotopy ring maps $B(0) \to K_{k,\Gamma}$ can exist only if $\Lambda^h \Gamma \cong \Gm$.  Here, we prove a converse to this fact.  More generally, we prove that if $\Lambda^h \Gamma \cong \Gm$, then the space of $\Einfty$-ring homomorphisms from $B(0)$ to $E_{k,\Gamma}$ is discrete and nonempty.  Our results are a generalization of \cite[Theorem 3.22]{Westerland}, and our proofs are highly inspired by Westerland's arguments.

\begin{cnv}
We fix throughout this \Cref{subsec:structuredB0} a height $h$ formal group $\Gamma$ over a perfect field $k$ of characteristic $p$.
For simplicity, we denote $K_{k,\Gamma}$ by $K$ and $E_{k,\Gamma}$ by $E$. 

Since there is an equivalence of categories between formal groups over $E_0$ and formal groups over the graded ring $E_*$, we will sometimes use the same notation to refer to an object of either category.  We adopt a similar convention for formal groups over $K_0$ and $K_*$.
\end{cnv}

\begin{rec}\label{HL3.4.1}
Hopkins and Lurie \cite{HopkinsLurie} show that there is a non-canonical isomorphism 
\[ E^*(\KZhpo) \cong E_{-*}\llbracket x \rrbracket.\]
Thus, $\mathrm{Spf} (E^*( \KZhpo))$ is a formal group over the ring $E_* \cong W(k)\llbracket u_1,u_2,\cdots,u_{h-1}\rrbracket[u^{\pm}]$.
Reducing this formal group mod the maximal ideal $\mathfrak{m}$, one obtains the formal group $\mathrm{Spf} (K^*(\KZhpo))$ over $K_*$, which was identified in the previous section as $\Lambda^h \Gamma$.
We therefore denote the formal group $\mathrm{Spf} (E^*(\KZhpo))$ by $\detGammatilde$.  It is a deformation of $\Lambda^h \Gamma$.

Though we will not need it here, Hopkins and Lurie in fact explicitly identify $\widetilde{\Lambda^h \Gamma}$ \cite[Theorem 3.4.1]{HopkinsLurie}.   Specifically, they describe the torsion on $\widetilde{\Lambda^h \Gamma}$ as Cartier dual to certain group schemes of alternating maps \cite[Section 3.2, Corollary 3.5.4]{HopkinsLurie}.
\end{rec}

Our goal in \Cref{subsec:structuredB0} is to prove the following result:

\begin{thm}
\label{Einfty-B(0)-orientations}
Let $\Hom(\widetilde{\Lambda^h \Gamma}, \Gm)$ denote the set of maps $\widetilde{\Lambda^h \Gamma} \to \Gm$ of formal groups  over $E_*$, and let $\Isom(\widetilde{\Lambda^h \Gamma}, \Gm) \subset \Hom(\widetilde{\Lambda^h \Gamma}, \Gm)$ denote the subset of isomorphisms.  Then, there are equivalences of spaces
\begin{align*}
\EinftyAlg(\suspinftypl\KZhpo,E)&\simeq \Hom(\widetilde{\Lambda^h \Gamma}, \Gm) \quad\text{and}\\
\EinftyAlg(B(0),E)&\simeq \Isom(\widetilde{\Lambda^h \Gamma}, \Gm).
\end{align*}
In particular, the spaces $\EinftyAlg(\suspinftypl\KZhpo,E)$ and $\EinftyAlg(B(0),E)$ are discrete.
\end{thm}

We will prove this by applying Goerss-Hopkins obstruction theory in the following form:

\begin{thm}[\cite{GoerssHopkins}, \cite{Piotr}]\label{thm:goersshopkins}
Let $A$ denote an augmented $\E_{\infty}$-algebra over $E$ such that $\pi_*(A \otimes_{E} K)$ is concentrated in even degrees, and suppose that $\pi_0(A \otimes_E K)$ is a perfect ring.  Then the natural map
$$\E_{\infty}\text{-}\mathrm{Alg}_E(A, E) \to E_*\text{-}\mathrm{Alg}(A_*, E_*)$$
is an equivalence of spaces.  In particular, the former space is discrete.
\end{thm}

\begin{proof}
This is standard, and follows for example from \cite[Corollary 5.11]{Piotr} (see also \cite[Theorem 5.10]{WesterlandSati} or the proof of \cite[Theorem 3.22]{Westerland} for similar statements).
There is a spectral sequence converging to the homotopy groups of the space of $\E_{\infty}$-ring homomorphisms from $A$ to $E$, pointed by the augmentation.  In this spectral sequence, $E^2_{0,0}$ is the set of $E_0$-algebra maps $A_0 \to E_0$ (which are the same as $E_*$-algebra maps from $A_*$ to $E_*$).  The other terms of the $E^2$ page vanish so long as the cotangent complex $L^{\E_{\infty}}_{\pi_0(A)/\pi_0(E)}$ is contractible.  The perfection hypothesis implies that this cotangent complex indeed vanishes (for example, see \cite[Proof of Theorem 7.7]{Piotr}, \cite[Corollary 21.3]{RezkHM}, \cite[Proposition 7.4]{GoerssHopkins}).
\end{proof}

We will apply this in the case $A = L_{K(n)}(E\otimes \Sigma^{\infty}_+\KZhpo).$  Note that the homotopy groups $\pi_* A$ are traditionally denoted
\[\pi_* A = E_*^{\wedge}(\KZhpo). \]
We analyze the structure of $E_*^{\wedge}(\KZhpo)$ in \Cref{RW3.4}, and we check that $K_*(\KZhpo)$ is perfect in \Cref{K0B0-is-perfect}.

\begin{ntn}
For the remainder of this section, we fix an isomorphism
$$E^{*}(\KZhpo) \cong E_*\llbracket x \rrbracket$$
of $E_*$-algebras (cf. \Cref{HL3.4.1}).
There is a Hopf algebra pairing 
\[E^{*}(\KZhpo)\otimes \Estarcts(\KZhpo)\to E_*,\]
which we denote by $\langle\text{--},\text{--}\rangle$.  We choose elements $b_i\in \Estarcts(\KZhpo)$ such that
\[\langle x^j, b_i\rangle =\begin{cases}
1 & i=j\\
0 & i\neq j,
\end{cases}\]
so that there is an isomorphism of $E_*$-modules 
\[\Estarcts(\KZhpo)\cong E_*\{b_0,b_1,\ldots\}_{\mfrak}^{\wedge}.\]  The elements $b_i$ lift the $K_*$-module generators of $K_*(\KZhpo)$.  Since these $K_*$-module generators are in even degrees, their lifts generate $\Estarcts(\KZhpo)$ as indicated \cite[Proposition 8.4(f)]{HovStrick}.  
\end{ntn}

\begin{lem}
\label{RW3.4}
Define the generating function $$b(s)=\sum_i b_i s^i \in E_*^{\wedge}(\KZhpo)\llbracket s\rrbracket.$$ Then $b_0 =1$ and $$b(s)b(t) =  b(s+_{\detGammatilde} t).$$
\end{lem}

\begin{proof}
Our argument is analogous to the proof of \cite[Theorem 3.4]{Ravenel-Wilson}; see also the proof of \Cref{Westerland5.5-strengthened}.

Define constants $a^n_{ij}$ such that $b_i b_j = \sum_n a^n_{ij}b_n$.  It follows that
$$ \langle \Delta(x^n), b_i\otimes b_j\rangle =  \langle x^n, b_ib_j\rangle = a_{ij}^n,$$ where $\Delta$ denotes the comultiplication on $E^*(\KZhpo)$.   In particular, $$\Delta(x^n) = \sum a^n_{ij} x^i \otimes x^j,$$ so the constants $a^1_{ij}$ are coefficients for the formal group law of $\detGammatilde$ under the coordinate $x$.

To prove the first statement, we note that by unitality, 
\[\langle  \Delta(x^n), b_i\otimes b_0\rangle = a_{i0}^n =\begin{cases}
1 & i=n\\
0 & \text{otherwise}.
\end{cases}\]
Thus $b_0b_n=b_n$ and so $b_0=1$.

For the second statement, consider the equalities:
\begin{align*}
b(s)b(t) &= \sum b_i b_j s^it^j \\
&= \sum_{i,j}\left(\sum_{n}a_{ij}^nb_n\right) s^it^j\\
&= \sum_nb_n\sum_{i,j}a_{ij}^ns^it^j\\
&=\sum_{n} b_n \left(\sum a_{ij}^1s^it^j\right)^n\\
&=\sum_n b_n(s+_{\detGammatilde}t)^n\\
&=b(s+_{\detGammatilde}t).
\end{align*}

\end{proof}

Next, we check the perfection hypothesis of \Cref{thm:goersshopkins}:
\begin{lem}
\label{K0B0-is-perfect}
The ring $K_0(\KZhpo)$ is perfect.
\end{lem}

\begin{proof}
We will also denote by $b_i \in K_0(\KZhpo)$ the image of the generators $b_i \in \Estarcts(\KZhpo)$ (multiplied by an appropriate power of $u$ to be in degree 0).   This gives an isomorphism of coalgebras $K_0(\KZhpo)\cong k\{b_0,b_1,\ldots\}$ where $\Delta(b_i)=\sum_{j+k=i}b_j\otimes b_k$.

As in \Cref{RW3.4}, define the generating function $b(s) = \sum b_i s^i \in K_0(\KZhpo)\llbracket s\rrbracket$.
The formal group $\detGamma$ is of height one, so $$[p]_{\detGamma}(x) = a_1x^p + \sum_{i=2}^{\infty}a_ix^{ip},$$ where $a_i\in \pi_0(K)$ and $a_1$ is a unit.
Applying \Cref{RW3.4}, we see that $$b(s)^p = b([p]_{\detGamma}(s)).$$ Comparing the coefficients of $s^{ip}$, we find relations of the form \[
b_i^p=a_1^ib_i+\sum_{0\leq j<i} c_jb_j \]
for some $c_j \in K_0(\KZhpo)$.  
Now consider the filtration on $K_0(\KZhpo)$ as a $K_0$-module given by
\[K_0\{b_0\}\subseteq K_0\{b_0,b_1\}\subseteq \cdots .\]
The previous relation implies that the Frobenius is a filtered map and that on the associated graded, the Frobenius acts by $a_1^i$. Since $a_1$ is a unit, the Frobenius map is an isomorphism and $K_0(\KZhpo)$ is perfect.
\end{proof}

Applying \Cref{thm:goersshopkins}, we conclude that the natural map
\[ \E_{\infty}\text{-}\mathrm{Alg}_E(L_{K(n)}(E\otimes \Sigma^{\infty}_+ \KZhpo) , E) \to \EstarAlg( \Estarcts(\KZhpo), E_*)\]
is an equivalence.  We now compute this latter set of $E_*$-algebra maps.

\begin{lem}
\label{EstarAlgebra-maps-is-hom-to-Gm}
There is a bijection 
\[\EstarAlg(\Estarcts \KZhpo, E_{*})\cong \Hom(\detGammatilde, \Gm).\]
\end{lem}

\begin{proof}
While this is a formal consequence of Cartier duality, we will need the explicit algebraic form of the identification, so we spell it out here.  By \Cref{rmk:Estarmodmap}, an $E_*$-\emph{module} map $f\colon \Estarcts\KZhpo\to E_{*}$ is determined by where it sends the generators $b_i$.  Thus, the set of such $E_*$-module maps $f$ is in bijection with the set of power series $g(x) \in E_*\llbracket x\rrbracket$ by sending a map $f$ to the power series  $$g(x)=\sum f(b_i) x^i=: f(b(x)).$$

If $f$ is a ring homomorphism, we have by \Cref{RW3.4} that $g(0)=f(b_0)=1$ and
\begin{align*}
g(x)g(y)&=f(b(x))f(b(y))\\
&=f(b(x)b(y))\\
&=f(b(x+_{\detGammatilde}y)) \\
&=g(x+_{\detGammatilde}y),
\end{align*}
so $g-1$ is a homomorphism $\detGammatilde \to \Gm$ of formal groups over $E_*$. 

Conversely, if $g-1$ is a homomorphism $\detGammatilde \to \Gm$ of formal groups over $E_*$, then
\[f(b(x)b(y))=f(b(x+_{\detGammatilde}y))=g(x+_{\detGammatilde}y)=g(x)g(y)=f(b(x))f(b(y)).\]
Equivalently, we have $$\sum f(b_ib_j)x^iy^j = \sum f(b_i)f(b_j)x^iy^j,$$ and equating terms of the generating functions we see that $f(b_ib_j)=f(b_i)f(b_j)$.  Thus,  $f$ determines a map of $E_*$-algebras. 
\end{proof}

\begin{rmk}\label{rmk:Estarmodmap}
Since $\Estarcts \KZhpo$ is not free as an $E_*$-module, but rather the $\mathfrak{m}$-completion of the free $E_*$-module on the $b_i$'s, we justify here why an $E_*$-module map $\Estarcts \KZhpo \to E_*$ is determined by where it sends the generators $b_i$.  The key point is that $E_*$ is $\mathfrak{m}$-complete, so such a map is determined by its reduction modulo powers of $\mathfrak{m}$.
\end{rmk}

For the next lemma, recall from \Cref{rmk:KFphsplitting} the element 
\[
z: Z \to L_{K(n)} \Sigma^{\infty}_+ \KFph.
\]
where $Z$ is a $K(n)$-local Picard element.

\begin{lem} \label{lem:z-is-primitive}
The map
\[K_*z \colon K_*Z \to K_*(\KFph)\]
factors through the sub-$K_*$-module of coalgebra primitives in $K_*(\KFph)$.
\end{lem}

\begin{proof}
Consider the diagram
\[
\begin{tikzcd}[cramped, column sep=small]
K_* \KFph \arrow[r, "\phi_1"{yshift=1pt}, pos=0.4] &
K'_* \KFph
& \arrow[l,"\phi_2"' {yshift=1pt}, pos=0.35] K(h)_*\otimes  \KFph \\
K_* Z \arrow[r] \arrow[u, "K_*z"] & K'_* Z \arrow[r]\arrow[u, "K'_*z"] & K(h)_*Z \arrow[u, "K(h)_*z"] 
\end{tikzcd}
\]
where $K(h)$ denotes the $(2p^h-2)$-periodic Morava $K$-theory associated to the Honda formal group and we have set $K' = K\otimes K(h)$.  

Note that the maps $\phi_1$ and $\phi_2$ in the top row are coalgebra maps, and each of the groups in the bottom row are rank $1$ free modules over their respective coefficient rings -- this is clear for Morava $K$-theories and  true for $K'$ because it is a wedge of Morava $K$-theories.

In \Cref{Kh-homology-of-KFph} and \Cref{rmk:KFphsplitting}, we noted that the image of $K(h)_*z$ is generated by an element $f_0$ which is primitive in $K(h)_*\KFph$.  It follows that $\phi_2(f_0)$ is also primitive, and consequently the image of $K'_*z$ is a $1$-dimensional $K'_*$-submodule of primitives in $K'_*\KFph$.  On the other hand, since $\phi_1$ is injective, we can check whether an element in $K_*\KFph$ is primitive after applying $\phi_1$.  It follows from the commutativity of the left square that anything in the image of $K_*z$ is primitive, as desired.  

\end{proof}

\begin{lem}
\label{EstarAlgebra-maps-is-iso-to-Gm}
The bijection of \Cref{EstarAlgebra-maps-is-hom-to-Gm} restricts to a bijection
\[\EstarAlg(\Estarcts B(0), E_{*})\cong \Isom(\Lambda^h\Gamma, \Gm).\]
\end{lem}

\begin{proof}
Under the identification given in the proof of \Cref{EstarAlgebra-maps-is-hom-to-Gm}, the subset of isomorphisms of formal groups $$\Isom(\Lambda^h\Gamma, \Gm) \subset \Hom(\Lambda^h\Gamma, \Gm)$$ corresponds to the subset of maps $f\in \EstarAlg( \Estarcts (\KZhpo), E_*)$ that send $b_1$ to a unit in $E_*$.  

On the other hand, let us choose a generator $y \in \Estarcts (\KZhpo)$ of the image of the map $$z: \Estarcts(Z) \to\Estarcts(\KZhpo).$$  Then we have an isomorphism $\Estarcts ( B(0)) \cong \Estarcts( \KZhpo) [y^{-1}]$, independent of the choice of generator $y$.  The subset $$\EstarAlg( \Estarcts (B(0)), E_*)\subset \EstarAlg( \Estarcts (\KZhpo), E_*)$$ corresponds to the maps where $y$ is sent to a unit.

It therefore suffices to show that $y$ is sent to a unit if and only if $b_1$ is sent to a unit.  To see this, it is enough to show that the images of $y$ and $b_1$ under the reduction map $ \Estarcts (\KZhpo) \to K_*(\KZhpo)$ generate the same 1-dimensional $K_*$-subspace.  This is true because the image of $y$ is primitive by \Cref{lem:z-is-primitive}, but since $K^*(\KZhpo)$ is a power series ring on one variable, the primitives of $K_*(\KZhpo)$ form a $1$-dimensional $K_*$-subspace generated by the image of $b_1$.
 \end{proof}

Combining \Cref{thm:goersshopkins} with Lemmas \ref{RW3.4}, \ref{K0B0-is-perfect}, \ref{EstarAlgebra-maps-is-hom-to-Gm}, \ref{lem:z-is-primitive}, and \ref{EstarAlgebra-maps-is-iso-to-Gm}, we have now proved \Cref{Einfty-B(0)-orientations}.  It now follows that an $\mathbb{E}_\infty$-ring homomorphism 
\[B(0) \to E_{k,\Gamma}\]
exists so long as $\widetilde{\Lambda^h \Gamma} \cong \Gm$.  By the following proposition, this condition is easily checkable:

\begin{prop}\label{lem:Gmrigid}
There is an isomorphism $\detGammatilde \cong \Gm$ of formal groups over $E_*$ if and only if there is an isomorphism $\detGamma \cong \Gm$ of formal groups over $K_*$.  
\end{prop}

\begin{proof}
If $\detGammatilde\cong (\Gm)_{E_*}$, then reducing this isomorphism modulo $\mfrak$ yields an isomorphism $\detGamma\cong(\Gm)_{K_*}$. 
Conversely, suppose $\gamma \colon \detGamma\to(\Gm)_{K_*}$ is an isomorphism.  Then, the pair $(\detGammatilde, \phi)$ determines a deformation of $(\Gm)_{K_*}$.  Since the universal deformation of $(\Gm)_{K_*}$ is $(\Gm)_{W(k)[u^{\pm}]}$, we obtain a map of formal groups $$(f,\theta)\colon (W(k)[u^{\pm}],\Gm)\to (E_*,\detGammatilde)$$ where $f\colon W(k)[u^{\pm}]\to E_*$ is a graded ring homomorphism and $\theta\colon f^*\Gm\to \detGammatilde$ is an isomorphism of formal groups. Since $f^*(\Gm)_{W(k)[u^{\pm}]} = (\Gm)_{E_*}$ we see that there is an isomorphism $\detGammatilde\cong (\Gm)_{E_*}$ of formal groups over $E_*$.
\end{proof}

\subsection{$R_1$-orientations of Morava $E$-theories at height $2$}
\label{sub:r1}

In the preceding sections, we discussed when a Morava $E$-theory receives a $B(0)$-orientation.  
The goal of this section is to discuss when a $B(0)$-orientation extends to an $L_{K(2)}R_1$-orientation.  We have the following result: 

\begin{thm} \label{thm:LiftsExist}
Fix an odd integer $3 \le n \le 2p-1$.  Let $A$ denote a $K(2)$-local $\E_n$-ring spectrum with trivial odd homotopy groups (e.g., $A$ could be a height $2$ Morava $E$-theory).  Suppose that there is an $\E_n$-ring homomorphism
\[
\begin{tikzcd}
L_{K(2)} \suspinftypl \KZthree[z^{-1}] \arrow{r} & A.
\end{tikzcd}
\]
Then there exists a dotted arrow completing the following diagram in the category of $\E_{n-1}$-algebras:
\[
\begin{tikzcd}
\suspinftypl \KZthree \arrow{r} \arrow{d} & L_{K(2)} \suspinftypl \KZthree[z^{-1}] \arrow{r} &  A \\
\suspinftypl\OS{\BPn<1>}{2p+2}\arrow[dashed,end anchor=south west]{urr}
\end{tikzcd}
\]
In particular, there is an $\E_{n-1}$-ring homomorphism
\[
\begin{tikzcd}
R_1 \arrow{r} & A.
\end{tikzcd}
\]
\end{thm}

\begin{proof}
The final part of this theorem follows by combining the first part with \Cref{thm:BocksteinMain}.  It remains to demonstrate the existence of the dashed arrow.  The assumed $\E_n$-ring map $L_{K(2)}\suspinftypl \KZthree[z^{-1}] \to A$ is determined by an $n$-fold loop map
\[\Krm(\Z,3) \to \GL(A).\]
Delooping this $n$ times yields a map of pointed spaces
\[\Krm(\Z,3+n) \to \B^n\GL(A).\]
Note that, because $n$ is odd and at most $2p-1$, $\OS{\BPn<1>}{3+n}$ has even cells.  On the other hand, $\B^n\GL(A)$ has only odd homotopy groups,
and so the following diagram commutes in the homotopy category of pointed spaces:
\begin{align*}
\begin{tikzcd}[ampersand replacement=\&]
\OS{\BPn<1>}{3+n} \arrow{r} \arrow{d} \& * \arrow{d} \\
\Krm(\Z,3+n) \arrow{r} \& \B^n\GL(A).
\end{tikzcd}
\tag{$1$}
\end{align*}
Choosing an explicit homotopy filling ($1$), we may compare the homotopy fibers along the two vertical arrows to obtain a map
\[
\begin{tikzcd}
\hofib\left( \Omega^{\infty} \Sigma^{3+n} \BPn<1> \to \Krm(\Z,3+n)  \right)
= \Omega^{\infty} \Sigma^{3+n+2p-2} \BPn<1>\arrow{r} & \B^{n-1}\GL(A).
\end{tikzcd}
\]
Looping $n-1$ times gives the desired map.
\end{proof}

\begin{rmk} As pointed out by the anonymous referee, the following more general statement is true, with an analogous proof: Suppose that $n,h,$ and $i$ are integers such that $n\leq 2p^{i+1} + i+1-h$, $n$ is odd and $h-i$ is even. If $A$ is an even $K(h)$-local $\mathbb{E}_n$ ring receiving an $\mathbb{E}_n$ $B(i)$ orientation, then that $B(i)$ orientation automatically extends to an $\mathbb{E}_{n-1}$ $B(i+1)$ orientation in at least one way.  The inequalities on $n,h$ and $i$ here are such that $\BP\langle i+1 \rangle_{2\nu(i)+h+n-i-1}$ is even. When $i$ does not have the same parity as $h$, it is not clear whether $B(i)$ orientations extend to $B(i+1)$ orientations.
\end{rmk}

We saw in \Cref{Westerland5.5-strengthened} that, if a Morava $K$-theory $K_{k,\Gamma}$ receives a homotopy $B(0)$-orientation, then the top exterior power of $\Gamma$ must be isomorphic to $\Gm$.  
A corollary of \Cref{thm:LiftsExist} is that, at height $2$, this condition is necessary and sufficient for the corresponding Morava $E$-theory to receive an $R_1$-orientation.  

\begin{cor} \label{cor:homotopy-orientations}
Let $\EkG$ denote the Morava $E$-theory corresponding to a height $2$ formal group $\Gamma$ over a perfect field $k$ of characteristic $p$.  Then there exists a homotopy commutative ring map
\[R_{1} \to \EkG\]
if and only if $\Lambda^2 \Gamma \cong \Gm$.
\end{cor}

\begin{proof}
Given a homotopy commutative ring map $R_1 \to E_{k, \Gamma}$, we obtain a homotopy ring map $B(0) \to K_{k, \Gamma}$.  By \Cref{Westerland5.5-strengthened}, this implies that $\Lambda^2 \Gamma \cong \Gm$.  

Conversely, suppose that $\Lambda^2 \Gamma \cong \Gm$.  By \Cref{lem:Gmrigid}, 
it is also true that $\widetilde{\Lambda^2 \Gamma} \cong \Gm$.  It follows by \Cref{Einfty-B(0)-orientations} that there exists an $\E_{\infty}$-ring map $B(0) \to E_{k, \Gamma}$.  Finally, applying \Cref{thm:LiftsExist}, we conclude that there is an $\E_{2p-2}$-ring map $R_1 \to E_{k, \Gamma}$.  In particular,  there is a homotopy commutative ring map.
\end{proof}

\begin{qst}
Let $\Gamma$ denote a height $h$ formal group over a perfect field $k$ of characteristic $p$, and suppose that $\Lambda^h \Gamma \cong \Gm$.  Must the Morava $E$-theory $\EkG$ receive an $\Einfty$-ring homomorphism from $R_{h-1}$?  In particular, is this true when $h=2$?
\end{qst}

In the next section we will explore structured orientations of $\EkG$ in the special case that the formal group $\Gamma$ arises from completing a supersingular elliptic curve.  In this case there is a canonical isomorphism $\Lambda^2 \Gamma \cong \Gm$, given by the Weil pairing of the elliptic curve. 

\section{Elliptic orientations} \label{sec:EllipticOrientation}

\begin{dfn}
An \emph{elliptic Morava $E$-theory} is the data of a triple $(E, X , f)$ where $E=E_{k,\Gamma} $ is a height $2$ Morava $E$-theory, $X$ is an elliptic curve over $E_0(*)$, and 
\[f\colon \mathrm{Spf} E^0(\mathbb{CP}^{\infty}) \cong \widehat{X}\]
is an isomorphism of the formal group of $E$ with the formal completion of $X$ at the identity.
\end{dfn}

As we explore below, the theory of \Cref{sec:Orientation} has particularly interesting consequences in the case of elliptic Morava $E$-theories.  This is powered by the Ando--Hopkins--Strickland--Rezk $\sigma$-orientation \cite{AHS, Ando-Hopkins-Rezk}, which associates to any elliptic Morava $E$-theory $E$ a canonical $\Einfty$-ring homomorphism
\[\sigma_E\colon \MUsix \to E.\]

\begin{ntn*}
For the remainder of the section, we fix a perfect field $k$ of characteristic $p$.  We additionally fix an elliptic Morava $E$-theory $(E,X,f)$, which we will sometimes refer to simply as $E$.  The reduction of $X$ modulo the maximal ideal is a supersingular elliptic curve over $k$, which we will denote by $C$.  We will implicitly identify $\widehat{C}$ with the formal group $\Gamma$ defining $E$ via the isomorphism $f$.
\end{ntn*}

\begin{cnstr}\label{cnstr:sigmanull}
The map $\sigma_E$ can be identified with a nullhomotopy of the composite of infinite loop maps \cite{ABGHR}
\[\BUsix \to \BSU \to \BU \stackrel{J}{\to} \BGLS \to \BGLE.\]
This nullhomotopy deloops to a nullhomotopy of a map of spectra $\Sigma^6 \ku \to \Sigma \glE$, which corresponds to a specific choice of dotted map in the diagram below:
\[
\begin{tikzcd}
\Sigma^6 \ku \arrow{r}{\beta} & \Sigma^4 \ku \arrow{r}{J} \arrow{d} & \Sigma \glS \arrow{r} & \Sigma \glE \\
& \Sigma^4 \HZ \arrow[urr, dashed]
\end{tikzcd}
\]
because $\Sigma^4 \HZ$ is the cofiber of the map $\beta\colon \Sigma^6 \ku \to \Sigma^4 \ku$.  By the universal property of $\glE$ \cite{ABGHR}, the resulting map $\Sigma^4 \HZ \to \Sigma \glE$ of spectra is the same data as an $\Einfty$-ring homomorphism
\[
w_{E}\colon \suspinftypl \KZthree \to E.
\]
\end{cnstr}

\begin{thm}\label{prop:agprop}
After $K(2)$-localization, the map $w_{E}$ constructed above sends $z$ to a unit in $\pi_{\star}(E)$,
and therefore induces a map of $\E_{\infty}$-ring spectra
\[w_E\colon L_{K(2)}\suspinftypl\KZthree [z^{-1}]=B(0) \to E.\]
\end{thm}

Most of this section, beginning with \Cref{sub:weil}, is dedicated to the proof of \Cref{prop:agprop}; our arguments are simply an unpacking of work of Ando, Hopkins, Rezk, and Strickland \cite{AHS, Ando-Hopkins-Rezk}.  Before introducing the algebraic geometry needed to prove \Cref{prop:agprop}, we give the topological consequences:

\begin{cor} \label{cor:NiceLifts}
In the $\infty$-category of $\Einfty$-ring spectra, the space $\EinftyAlgw(R_1, E)$ of lifts in the diagram
\[
\begin{tikzcd}
B(0) \arrow{r}{w_E} \arrow{d} & E \\
L_{K(2)}R_1 \arrow[dashed]{ru}
\end{tikzcd}
\]
is a collection of path components inside the space of $\Einfty$-ring maps $R_1 \to E$.
\end{cor}

\begin{proof}
This is a consequence of the fact that the space of $\Einfty$-ring maps from $B(0)$ to $E$ is discrete (\Cref{Einfty-B(0)-orientations}).  Indeed, the desired space of lifts can be described as a homotopy pullback
\[
\begin{tikzcd}
\EinftyAlgw(R_1, E) \arrow{r} \arrow{d} & \EinftyAlg(R_1,E) \arrow{d} \\
* \arrow{r} & \EinftyAlg(B(0),E),
\end{tikzcd}
\]
where the bottom arrow picks the point $w_E$ out of the discrete set $\EinftyAlg(B(0),E)$. We recall that, since $E$ is $K(2)$-local, the space of $\mathbb{E}_{\infty}$-ring maps from $L_{K(2)} R_1$ to $E$ is canonically identified with the space of $\mathbb{E}_{\infty}$-ring maps from $R_1$ to $E$.
\end{proof}

\begin{dfn} \label{dfn:WeilSpace}
We refer to the space $\EinftyAlgw(R_1, E)$ of \Cref{cor:NiceLifts} as \emph{the space of $\Einfty$-ring maps $R_1 \to E$ that respect the Weil pairing}.
\end{dfn}

\begin{thm} \label{thm:elliptic-orientations-main}
The space $\EinftyAlgw(R_1, E)$ of $\Einfty$-ring maps $R_1 \to E$ that respect the Weil pairing is naturally homotopy equivalent to
\begin{enumerate}
\item If $p=2$, the space of $\Einfty$-ring homomorphisms $\MSU \to E$.
\item If $p=3$, the space of $\Einfty$-ring homomorphisms $\MSpin \to E$.
\item If $p>3$, the space of $\Einfty$-ring homomorphisms $\suspinftypl W_{1,4} \to E$.
\end{enumerate}
\end{thm}

\begin{proof}
Consider the diagram
\[
\begin{tikzcd}
\Sigma^{4} \BPn<1> \arrow{r} \arrow{d} & \Sigma^4 \ku \arrow{r}{J} & \Sigma \glS \arrow{r} & \Sigma \glE \\
\Sigma^4 \HZ  \arrow[rrru, end anchor=-165] \arrow[d] \\
\Sigma^{2p+3} \BPn<1>, \arrow[dashed,uurrr, end anchor=-150, shorten >=7pt]
\end{tikzcd}
\]
where the leftmost horizontal arrow is the inclusion of an Adams summand for $p>2$ and the canonical equivalence when $p=2$.  By the universal property of $\glE$ \cite{ABGHR}, the space of lifts indicated by the dashed arrow is equivalent to the space of lifts
\[
\begin{tikzcd}
\suspinftypl \KZthree \arrow{r}{w_E} \arrow{d} & E \\
\suspinftypl W_1 \arrow[dashed]{ur}
\end{tikzcd}
\]
in the $\infty$-category of $\Einfty$-rings.  By \Cref{prop:agprop} this is the same as the space of lifts
\[
\begin{tikzcd}
B(0) \arrow{r}{w_E} \arrow{d} & E, \\
L_{K(2)}R_1 \arrow[dashed]{ur}
\end{tikzcd}
\]
which we have defined as $\EinftyAlgw(R_1, E)$.

Since the vertical sequence
\[\Sigma^4 \BPn<1> \to \Sigma^4 \HZ \to \Sigma^{2p+3} \BPn<1>\]
is a fiber sequence of spectra, the space $\EinftyAlgw(R_1, E)$ is naturally equivalent to the space of nullhomotopies of the composite
\[\Sigma^4 \BPn<1> \to \Sigma^4 \ku \stackrel{J}{\to} \Sigma \glS \to \Sigma \glE.\]
By the theory of Thom spectra \cite[Corollary 4.12]{ABGHR1}, this is in turn equivalent to the space of $\Einfty$-ring homomorphisms
\[W_{1,4}^{\gamma} \to E,\]
where $W_{1,4}^{\gamma}$ is the Thom spectrum of the spherical bundle $W_{1,4} \to \BGLS$ obtained by applying $\Omega^{\infty}$ to the map
$\Sigma^4 \BPn<1> \to \Sigma \glS$.

To finish the proof, it suffices to identify $W_{1,4}^{\gamma}$ when $p=2$, when $p=3$, and when $p>3$:
\begin{itemize}
\item When $p=2$, $\BPn<1> \simeq \ku $, so $W_{1,4}\simeq \OS{\ku}{4} \simeq \BSU$ and the Thom spectrum of $J\colon \BSU\to \BGLS$ is $W_{1,4}^{\gamma} \simeq \MSU$.
\item When $p=3$, $\BPn<1>\simeq \ko$, so $W_{1,4} \simeq \OS{\ko}{4} \simeq \BSpin$ and the Thom spectrum of $J\colon \BSpin\to\BGLS$ is $W_{1,4}^{\gamma} \simeq \MSpin$.
\item When $p>3$, there is a splitting $\ku=\bigvee_{i=0}^{p-2}\susp^{2i}\BPn<1>$. Let $l$ denote an integer that reduces to a primitive $(p-1)$st root of unity modulo $p$. By the stable Adams conjecture \cite{Friedlander, BhatKitch}, if $\psi^l$ is the Adams operation corresponding to $l$, then the composite $J\circ (\psi^l-1)\colon \ku \to \ku \to \Sigma \glS$ is null.
The map $\psi^l-1$ is also an isomorphism on the summands $\Sigma^{2i} \BPn<1>$ for which $i$ is not divisible by $p-1$.  It follows that, on these summands, the $J$-homomorphism is null.  In particular, when $p>3$, the map $\Sigma^4 \BPn<1> \to \Sigma \glS$ is null, and so there is an equivalence of $\E_{\infty}$-ring spectra $W_{1,4}^{\gamma}\simeq \suspinftypl W_{1,4}$. 
\end{itemize}
\end{proof}
\begin{rmk}
We conjecture that every elliptic Morava $E$-theory at the prime $2$ receives at least one $\Einfty$-ring homomorphism from $\MSU$, and that every elliptic $E$-theory at $p=3$ receives an $\Einfty$-ring map from $\MSpin$.
At the prime $3$, Dylan Wilson \cite{Wilson-Orientations} has proved that the Ochanine genus factors through an $\Einfty$-ring homomorphism
\[\MSpin \to \tmf_0(2).\]
In forthcoming joint work with Andrew Senger, the middle author will produce an $\Einfty$-ring homomorphism
\[\MSU \to \tmf_1(3)^{\wedge}_2.\] 
\end{rmk}

\begin{rmk}
The equivalence of Theorem \ref{thm:elliptic-orientations-main} can be described fairly explicitly, using the fact that $\pi_0\EinftyAlg(\MUsix,E)$ is a torsor (see, e.g., \cite[Section 2.3]{Ando-Hopkins-Rezk}) for the abelian group
\[\pi_0\EinftyAlg(\suspinftypl \BUsix,E) \cong \pi_0\Spectra(\tau_{\ge 6} \KU,\glE).\]
Consider for example the case when $p=2$.  Given an $\Einfty$-ring homomorphism
\[f\colon \MSU \to E,\]
we can use the canonical map $\MUsix \to \MSU$ to pull this back to a map $f'\colon\MUsix \to E$.
The difference of $f'$ and the $\sigma$-orientation is then a well defined $\Einfty$-ring homomorphism
\[f'-\sigma_E \colon \suspinftypl \BUsix \to E.\]
The bottom cell $x_6 \in \pi_6(\BUsix)$ will map to an invertible class in $\pi_6 E$, and thus there is a unique extension of $f'-\sigma_E $ to a map
\[f'-\sigma_E \colon \suspinftypl \BUsix[x_6^{-1}]=R_1 \to E.\]
\end{rmk}

\subsection{Weil pairings}\label{sub:weil}
Our goal in this section is to prove \Cref{prop:agprop} assuming one statement (\Cref{prop:weilag}) that we address in the next section.

We start by making a reduction from Morava $E$-theory to Morava $K$-theory.  Let $K=K_{k,\Gamma}$ denote the Morava $K$-theory associated to $E$ and let $q\colon  E \to K$ denote the associated homotopy ring map (cf. \Cref{dfn:moravaek}).  

\begin{lem}\label{lem:picinvert}
In the situation of \Cref{prop:agprop}, it suffices to show that the composite \[Z\xrightarrow{z} L_{K(2)}\suspinftypl \KZthree \xrightarrow{w_E}E \xrightarrow{q} K\]
is a unit in $\pi_{\star}(K)$.
\end{lem}
\begin{proof}
Observe that an element $\gamma \in \pi_0(E)$ is invertible if and only if the composite $q\gamma \in \pi_0(K)$ is invertible.  The lemma then follows by noting that $z$ is a unit if and only if $z^{p-1}$ is a unit, and $Z^{\otimes p-1} \simeq S^0$ (cf. \Cref{rmk:KFphsplitting}).
\end{proof}

Thus, in the remainder of the section, we will work with Morava $K$-theory and analyze the map $w_K$ defined as the composite
\[w_K\colon  \suspinftypl \KZthree \xrightarrow{w_E} E \xrightarrow{q} K.\]
Work of Ando and Strickland \cite{AS} relates such ring maps $\suspinftypl \KZthree \to K$ to a structure on the formal group $\Gamma$ known as a \emph{Weil pairing}.  We briefly recall their work:

\begin{dfn}\label{dfn:weilpairing}
A \emph{Weil pairing} on $\Gamma$ is a collection of bilinear pairings
\[\beta_i \colon  \Gamma[p^i] \times \Gamma[p^i] \to \mu_{p^i}\]
for each $i\geq 0$, such that for any $x,y\in \Gamma[p^{i+1}]$, one has the relations \[\beta_i(px,py) =\beta_{i+1}(x,y)^p\]
and
\[\beta_{i+1}(x,y) = \beta_{i+1}(y,x)^{-1}.\]
The Weil pairings on $\Gamma$ can be organized into a group scheme over $k$, which we denote by $W(\Gamma)$ following \cite{AS}.
\end{dfn}

\begin{exm}\label{exm:weilell}
The Weil pairing of the elliptic curve $C$ determines a Weil pairing on the formal group $\Gamma \cong \widehat{C}$ of $C$ in the sense of \Cref{dfn:weilpairing}.  We refer to this simply as ``the canonical Weil pairing on $\Gamma$." 
\end{exm}

\begin{prop}[\cite{AS}, Proposition 2.9]\label{prop:weilkz3}
There is an isomorphism
\[b\colon  \Spec (K_0(\KZthree)) \to W(\Gamma)\]
of group schemes over $k$.
\end{prop}

A homotopy ring map $\suspinftypl \KZthree\to K$ induces a map of rings $$K_0(\suspinftypl \KZthree) \to \pi_0(K)\cong k,$$ which determines a $k$-point of the scheme $\Spec(K_0(\KZthree))$.  By \Cref{prop:weilkz3}, this in turn gives rise to a Weil pairing.  By design, the Weil pairing corresponding to $w_K$ is canonical:

\begin{prop}\label{prop:weilag}
Under the isomorphism $b$ of \Cref{prop:weilkz3}, the map 
\[w_K\colon  \suspinftypl \KZthree  \to K\]
corresponds to the canonical Weil pairing on $\Gamma$ (up to a possible sign).
\end{prop}

\begin{rmk}\label{rmk:sign}
The work we build on to prove \Cref{prop:weilag} leaves a possible ambiguity in the sign \cite[Remark 4.3]{AS} which we do not resolve here and will not be important for our theorems.
\end{rmk}

We defer the proof of this proposition to the next section and turn instead to the proof of \Cref{prop:agprop} assuming \Cref{prop:weilag}.

\begin{proof}[Proof of \Cref{prop:agprop}]
Combine \Cref{lem:picinvert} and the following corollary of \Cref{prop:weilag}.
\end{proof}

\begin{cor}\label{cor:weildegen}
The map $w_K\colon  \suspinftypl \KZthree \to K$ sends the element 
\[z\in \pi_Z(L_{K(2)}\suspinftypl \KZthree) \]
to a unit in $\pi_{\star}( K)$.
\end{cor}
\begin{proof}
Consider the $K_*$-module map $w\colon  K_*Z \to K_*$ induced by the composite $w_Kz \colon  Z \to K$.  Since $K_*Z$ is a free $K_*$-module of rank $1$, it suffices to show that this map $w$ is nonzero.

Recall from \Cref{sub:Z} that the element $z$ factors as a composite
 \[Z\to L_{K(2)} \suspinftypl \KFptw\to L_{K(2)}\suspinftypl \KZthree, \]
where the first map extends to a $K(2)$-local splitting
\[L_{K(2)}\suspinftypl \KFptw \simeq L_{K(2)} \Sbb \vee Z \vee Z^{\otimes 2} \vee \cdots \vee Z^{\otimes p-1}\]
such that $L_{K(2)} \suspinftypl \KFptw$ is generated as a homotopy ring by the inclusion of $Z$.

Suppose for the sake of contradiction that the map $w\colon  K_*Z \to K_*$ is zero.  Then, since the composite
\begin{equation}\label{eqn:weildegen}
K_*(\KFptw) \to K_*(\KZthree) \xrightarrow{w_k} K_*
\end{equation}
is a ring map, the previous remarks imply that it factors through the augmentation.
By \cite[\S 2]{AS}, the set of maps $K_*(\KFptw) \to K_*$ can be identified with what are called
\emph{$e_p$-pairings} on $\Gamma$, which are certain pairings
\[\Gamma[p] \times \Gamma[p] \to \Gm.\]
The map (\ref{eqn:weildegen}) factoring through the augmentation corresponds to the zero $e_p$-pairing.
On the other hand, by \Cref{prop:weilag}, the pairing arising from $w_K$ is the underlying $e_p$-pairing of the
canonical Weil pairing on an elliptic curve (up to a sign), which is nonzero, contradicting our assumption.
\end{proof}


\subsection{Cubical structures and the proof of \Cref{prop:weilag}}\label{sub:cub}


We give here a very brief recollection of the results of \cite{AHS} in order to extract a proof of \Cref{prop:weilag}; we do not intend for this to be a thorough exposition of the material and refer the reader to \cite{AHS, AS, Breen} for further details.

\begin{rec}[\cite{AHS} Definition 2.40]\label{rec:cubical}
Given a line bundle $\L$ on $\Gamma$, there is the notion of a \emph{cubical structure} on $\L$.  Let $\Theta^3(\L)$ denote the line bundle on $\Gamma^{\times 3}$ whose fiber at $(x,y,z)\in \Gamma^{\times 3}$ is given by the formula
\[\Theta^3(\L)_{(x,y,z)} = \frac{\L_{x+y} \otimes \L_{x+y}\otimes \L_{y+z} \otimes \L_0 }{\L_{x+y+z} \otimes \L_x \otimes \L_y \otimes \L_z}.\]
A cubical structure is a trivialization of $\Theta^3(\L)$ satisfying certain conditions.  We will denote the set of cubical structures on $\L$ by $C^3(\Gamma;\L)$.  These can be assembled into an affine group scheme over $k$, which we give the same name. 
\end{rec}

\begin{rec}[\cite{Breen} \S 2, see also \cite{Grieve} \S 5]\label{rec:cubicaltoweil}
Given a line bundle $\L$ on $\Gamma$ and a cubical structure $\gamma \in C^3(\Gamma; \L)$, one can extract a Weil pairing which we will denote by $\beta_{\L}(\gamma)$.  This procedure determines a map
\[\beta_{\L}\colon  C^3(\Gamma; \L)\to W(\Gamma)\]
of $k$-group schemes.

The association of a Weil pairing to a cubical structure is \emph{natural} in the sense that an isomorphism $\psi\colon  \L \simeq \L'$ of line bundles induces an identification $\psi_*\colon  C^3(\Gamma; \L) \to C^3(\Gamma\colon  \L')$ under which we have an equality of pairings $\beta_{\L}(\gamma) = \beta_{\L'}(\psi_* \gamma)$.
\end{rec}


These notions have incarnations in topology by the following results from \cite{AHS, AS}:

\begin{thm}[\cite{AS}, Theorem 1.1, Corollary 4.4]\label{thm:busixag}
There is an isomorphism
\[f\colon \Spec K_0(\BUsix ) \cong C^3(\Gamma; \Ocal_{\Gamma})\]
of group schemes over $k$ fitting into a square
\begin{equation*}
\begin{tikzcd}
\Spec K_0(\BUsix) \arrow[r]\arrow[d,"f"', "\cong"] & \Spec K_0(\KZthree) \arrow[d,"b"', "\cong"] \\
C^3(\Gamma ; \Ocal_{\Gamma}) \arrow[r,"\beta_{\Ocal_{\Gamma}}"'] & W(\Gamma).
\end{tikzcd}
\end{equation*}
which commutes up to a possible sign (cf. Remark \ref{rmk:sign}).
Here, the top map is induced by the map $\KZthree \to \BUsix$ in the Postnikov tower for $\BUsix$ and $b$ is the identification from \Cref{prop:weilkz3}.
\end{thm}

Consider the following two torsors for $\Spec K_0(\BUsix)$:
\begin{enumerate}
\item The $k$-scheme $\Spec K_0(\MUsix)$, which is a torsor for $\Spec K_0(\BUsix)$ because $\MUsix$ is a Thom spectrum over $\BUsix$ and the ring $K$ is complex orientable.
\item The $k$-scheme $C^3(\Gamma ;\Ical(0))$, which is a torsor for $C^3(\Gamma ;\Ocal_{\Gamma}) \cong \Spec K_0(\BUsix)$, where $\Ical(0)$ denotes the ideal sheaf of functions vanishing at the identity.  Roughly, the torsor structure comes from the bundle $\Ical(0)$ being trivializable and any two cubical structures on $\Ical(0)$ differing by a cubical structure on the trivial bundle (see \cite[Proposition 2.43]{AHS}).
\end{enumerate}

\begin{thm}[\cite{AHS}, Theorem 2.48]\label{thm:musixcub}
There is a map (and therefore, isomorphism) of torsors for the $k$-group scheme $\Spec K_0(\BUsix)$:
\[g\colon  \Spec K_0(\MUsix) \to C^3(\Gamma ;\Ical(0)).\]
\end{thm}

For the following, we will need the following definition/construction: 

\begin{dfn} \label{dfn:talpha}
    Let $R$ be a homotopy commutative ring and $\alpha: MU\to R$ be a homotopy ring map.  
    By \cite[\S 2.3]{Ando-Hopkins-Rezk} (more precisely, the analogue for homotopy rings), the space $\Map_{\hAlg}(\MUsix, R)$ is a torsor for $\Map_{\hAlg}(\Sigma^{\infty}_+ \BUsix,R)$ and so the composite $\MUsix \to \mathrm{MU} \xrightarrow{\alpha} R$ determines an isomorphism
    \[
    t_{\alpha}:  \pi_0 \Map_{\hAlg}(\MUsix, R) \to \pi_0 \Map_{\hAlg}(\BUsix, R)
    \]
    given by ``taking the difference with $\alpha$'' in the torsor structure.  
\end{dfn}

The key property of the $\sigma$-orientation that we use is the following:

\begin{rec}\label{rec:sigmaweil}
The sheaf $\Ical(0)$ on $\Gamma$ arises as the restriction of a sheaf on all of $C$, which we will denote by $\Ical_C(0)$.
By the theorem of the cube (see, for instance, \cite{MumAb}), the ideal sheaf $\Ical_C(0)$ acquires a unique cubical structure $\gamma_C$, which determines a cubical structure $\gamma \in C^3(\Gamma; \Ical(0))$ having the following properties:
\begin{itemize}
\item Consider the composite $\sigma_K\colon  \MUsix \xrightarrow{\sigma_E} E \to K$.  Then under the identification $g$ of \Cref{thm:musixcub}, $\sigma_K$ corresponds to the cubical structure $\gamma \in C^3(\Gamma; \Ical(0))$.
\item  Since the cubical structure $\gamma_C$ corresponds to the symmetric biextension determined by the Poincare bundle on $C$ (cf. \cite{MumAb, MumBiext} for the relevant background), the Weil pairing $\beta_{\Ical (0)}(\gamma)$ associated to the cubical structure $\gamma$ (cf. \Cref{rec:cubical}) is exactly the canonical Weil pairing on $\Gamma$.
\end{itemize}
\end{rec}

This concludes our recollections of \cite{AHS, AS}.  We now use these ideas to prove \Cref{prop:weilag}.  

\begin{proof}[Proof of \Cref{prop:weilag}]
Fix a homotopy commutative ring map $\alpha\colon MU \to E$.  
This determines a map $\MUsix \to MU \xrightarrow{\alpha} E \xrightarrow{q} K$, which determines via \Cref{thm:musixcub} a commutative square of isomorphisms
\begin{equation*}
\begin{tikzcd}
\Hom_{k}(K_0(\MUsix), k) \arrow[r,"\cong","t_{q\alpha}"']\arrow[d,"g"',"\cong"] & \Hom_{k}(K_0(\BUsix),k) \arrow[d,"f"',"\cong"]\\
C^3(\Gamma; \Ical(0))(k) \arrow[r,"\cong","s_{q\alpha}"'] & C^3(\Gamma; \Ocal_{\Gamma})(k),
\end{tikzcd}
\end{equation*}
where the map $s_{q\alpha}$ can be described in the language of \cite{AHS} as follows: the map $q\alpha\colon MU \to K$ provides a trivialization of the bundle $\Ical(0)$ and $s_{q\alpha}$ is induced by the resulting identification $\Theta^3(\Ical(0))\cong \Theta^3(\Ocal_{\Gamma})$.  Combining this square with \Cref{thm:busixag}, we obtain a commutative diagram
\begin{equation*}
\begin{tikzcd}
\Hom_{k}(K_0(\MUsix), k) \arrow[r,"\cong","t_{q\alpha}"']\arrow[d,"g"',"\cong"] & \Hom_{k}(K_0(\BUsix),k) \arrow[d,"f"',"\cong"] \arrow[r] & \Hom_{k}(K_0(\KZthree), k) \arrow[d,"\cong","b"']\\
C^3(\Gamma; \Ical(0))(k) \arrow[r,"\cong","s_{q\alpha}"'] & C^3(\Gamma; \Ocal_{\Gamma})(k) \arrow[r,"\beta_{\Ocal_{\Gamma}}"'] &W(\Gamma)(k)
\end{tikzcd}
\end{equation*}
Recall that we are interested in analyzing a map $\sigma_K\colon  \MUsix \to K$, which induces on $K$-homology a map of $k$-algebras $(\sigma_K)_*\colon  K_0(\MUsix) \to k$ corresponding to a point in the top left set.  Our strategy will be to determine its image in $W(\Gamma)(k)$ under the upper and lower composites.

We start with the lower composite.  Unwinding the definitions, the horizontal composite $\beta_{\Ocal_{\Gamma}} \circ s_{q\alpha}$ takes a cubical structure $\gamma$ on $\Ical(0)$, regards it as a cubical structure on $\Ocal_{\Gamma}$ via the trivialization of $\Ical(0)$ provided by $\alpha$, and then takes the associated Weil pairing.  By \Cref{rec:cubicaltoweil}, this coincides with directly taking the Weil pairing $\beta_{\Ical(0)}(\gamma)$, which is the canonical Weil pairing on $\Gamma$ by \Cref{rec:sigmaweil}.

It therefore suffices to show that the top composite
\[\Hom_{k}(K_0(\MUsix), k) \xrightarrow{t_{q\alpha}} \Hom_{k}(K_0(\BUsix),k) \to \Hom_{k}(K_0(\KZthree), k)\]
 takes $(\sigma_K)_*\colon  K_0(\MUsix) \to k$ to the map induced on $K$-homology by $w_K\colon  \suspinftypl \KZthree\to K$.  This follows from an analogous statement at the level of spectra and before reducing from $E$ to $K$, which we prove as \Cref{sigma-maps-to-w}.
\end{proof}

\begin{lem}
\label{sigma-maps-to-w}
Let $\alpha\colon \MU \to E$ as before, and let
\[t_{\alpha}\colon  \Map_{\hAlg}(\MUsix, E) \to \Map_{\hAlg}(\suspinftypl \BUsix, E)\]
be the map of \Cref{dfn:talpha}.  Here,  $\Map_{\hAlg}$ denotes the set of homotopy commutative ring maps.  Then, under the composite
\[
    \Map_{\hAlg}(\MUsix, E) \xrightarrow{t_{\alpha}} \Map_{\hAlg}(\suspinftypl \BUsix, E) \to \Map_{\hAlg}(\suspinftypl \KZthree, E),
\]
the $\sigma$-orientation $\sigma_E$ is sent to $w_E$.
\end{lem}
\begin{proof}
In order to use the theory of Thom spectra \cite{ABGHR}, it is convenient to lift $\alpha$ to an $\E_2$-ring map, which can always be done by \cite[Theorem 1.2]{ChadwickMandell}.  We then compute instead the composite
\begin{equation}\label{eqn:e2composite}
\Map_{\E_2}(\MUsix, E) \xrightarrow{t_{\alpha}} \Map_{\E_2}(\suspinftypl \BUsix, E) \to \Map_{\E_2}(\suspinftypl \KZthree, E)
\end{equation}
which is analogous to the composite in the lemma statement with homotopy ring maps replaced by $\E_2$-ring maps.  By the theory of Thom spectra \cite{ABGHR}, the $\E_2$-map $\sigma_E\colon  \MUsix \to E$ can be identified with a nullhomotopy of the composite $\BUsix \to \BSU \xrightarrow{J} \BGLE$ in $\E_2$-spaces.  Similarly, $\alpha\colon \MU \to E$ restricts to a map $\MSU \to \MU \xrightarrow{\alpha} E$ corresponding to a nullhomotopy (in $\E_2$-spaces) of $\BSU \xrightarrow{J} \BGLE$.  These induce a homotopy coherent diagram in $\E_2$-spaces
\begin{equation*}
\begin{tikzcd}
 \BUsix \arrow[rr] \arrow[dd]  & & \BSU \arrow[lldd, Leftrightarrow, "\sigma_E"] \arrow[rd, bend left=20]\arrow[dd,"J"',""{name=U}] & \\
  & & & * \arrow[ld, bend left=20]\arrow[Leftrightarrow, "\alpha", to=U]\\
 * \arrow[rr] & & \BGLE &
\end{tikzcd}
\end{equation*}
where the homotopies have been labeled by the corresponding orientations of $E$.  This extends to a larger diagram
\begin{equation*}
\begin{tikzcd}
\KZthree\arrow[dd,dashed,"\psi_2"]\arrow[rr, dashed] \arrow[rd, dashed, "\psi_1"]& &\BUsix \arrow[rr]\arrow[rd, dashed]   & & \BSU  \arrow[rd]& \\
&\GLE \arrow[ld, dashed, "\simeq"'] \arrow[rr, dashed]&  & \GLE \arrow[rr] \arrow[ld]& & * \arrow[ld]\\
\GLE \arrow[rr,dashed]& & * \arrow[rr] \arrow[from=uu, crossing over]& & \BGLE \arrow[from=uu, crossing over]&
\end{tikzcd}
\end{equation*}
where $\GLE$ is constructed as a pullback, and the remaining dotted arrows are constructed by taking fibers.  To finish, we note that $\psi_2\colon  \KZthree\to \GLE$ is adjoint to the map $w_E\colon  \suspinftypl \KZthree\to E$ of \Cref{cnstr:sigmanull} and $\psi_1$ is adjoint to the image of $\sigma_E$ under the composite (\ref{eqn:e2composite}) above.  The diagram shows that these are homotopic and the lemma is proved.
\end{proof}

\section{The upper bound on chromatic height}\label{sec:UpperBound}
Our goal in this section is to complete the proof of the following theorem:

\begin{thm}\label{thm:height}
Let $n$ and $h$ be nonnegative integers.  Then the $\Einfty$-ring spectrum
\[L_{K(n)}R_{h}\]
vanishes if and only if $n>h+1$.  That is, the ring $R_{h}$ has chromatic height $h+1$.
\end{thm}

By the main result of \cite{HinfBousfield}, this is equivalent to showing both that $L_{K(h+1)}R_{h} \not\simeq *$ and that $L_{K(h+2)}R_{h} \simeq *$.  The former statement is \Cref{cor:lowerbound}, so we will focus on the latter, which we will prove by showing that $K(n)_*(R_{h}) \cong 0$ for $n > h + 1$.   Throughout the proof, we will assume familiarity with the theory of Hopf rings, and, in particular, the paper \cite{Ravenel-Wilson}.  For the convenience of the reader, \Cref{sub:rwreview} lists some of the key formulas that we will use.

\subsection{The $K(n)$-homology of $R_h$}
Throughout this section, we fix nonnegative integers $n$ and $h$ with $n>h+1$.
By definition,
\[R_h = \suspinftypl \BPph [x_{2\nu(h)}^{-1}],\]
where $x_{2\nu(h)}\in \pi_{2\nu(h)}(\suspinftypl \BPph)$ denotes a bottom cell.  Thus, in order to show $K(n)_*(R_{h}) \cong 0$, it suffices to show that the Hurewicz image of $x_{2\nu(h)}$ is nilpotent.  We prove this as \Cref{prop:xpzero} below.

\begin{rmk}
Although we do not use it in our proof, this computation in $K(n)$-homology essentially follows from knowledge of  $\HFp_*(\BPph)$, which, as an algebra, is a tensor product of a divided power algebra on infinitely many generators and a polynomial algebra on infinitely many generators.  This computation of $\HFp$-homology is known to Ravenel and Wilson, but the authors were unable to find it in the literature.
\end{rmk}

\begin{ntn*}
For a Hopf ring $E_{*}\OS{F}{2*}$ we denote by $QE_*\OS{F}{2*}$ the $*$-indecomposables. This forms a graded ring under the $\circ$ product.
\end{ntn*}

\begin{thm}\label{prop:xpzero}
Let $u$ denote the Hurewicz image of $x_{2\nu(h)}$ in $K(n)_*(\BPph)$. Then $u^p = 0.$
\end{thm}

We first outline the strategy.  In the Ravenel-Wilson basis, the element $u$ is given by $b_{(0)}^{\c \nu(h)}$.   It follows from the distributivity law for Hopf rings that
\[u^p = \left(b_{(0)}^{\c \nu(h)}\right)^{*p} = [p]\circ b_{(1)}^{\c \nu(h)}.\]
We start by showing in \Cref{vh*b1^p^h} that, in the ring $QK(n)_*(\OSBPh{2*})$, the following relations hold:
\begin{align*}
[v_r]\circ b_{(1)}^{\c p^{r}}
&=-[v_{r+1}]\circ b_{(0)}^{\c p^{r+1}}\pmod{[I_r]}\quad\text{for $1\leq r <h-1$ and}\\
[v_{h-1}]\circ b_{(1)}^{\circ p^{j-1}} &= 0
\phantom{[v_{r+1}]\circ b_{(0)}^{\c p^{r+1}}}\pmod{[I_{h-1}]},
\end{align*}
where $[I_{r}]$ denotes the ideal $([p], [v_1], \ldots, [v_{r-1}])$.
Then we lift this to a statement before quotienting by $[I_r]$ using \Cref{Ih*b0^p^k=0}.  The relations for various $r$ link up to show that $u^p = [p]\c b_{(1)}^{\c \nu(h)}= 0$.

\begin{lem}
\label{bi^p^r}
%
%
Let $r<h$ be a nonnegative integer. In $QK(n)_*(\OSBPh{2*})$,
\[[v_r]\c b_{(0)}^{\c p^r} \in [I_r]\]
\end{lem}
\begin{proof}
The $p$-series of $K(n)$ is
\[[p]_{K(n)}(x) = O(x^{p^{h}}) \pmod{I_h}.\]
The $p$-series of $\BPh$ is
\[[p]_{\BPh}(x) = v_rx^{p^r} + O(x^{p^{r+1}})\pmod{I_r}\]
Thus, the Ravenel-Wilson relation $b([p]_{K(n)}(x)) = [p]_{[\BPh]}(b(x))$ says in this case that
\[O(x^{p^h}) = [v_r]\c \sum_{i} b_i^{\c p^{r}} x^{ip^{r}} + O(x^{p^{r+1}}) \pmod{[I_r], I_h}\]
We extract the $x^{p^r}$ term from both sides. Since $r<h$, the coefficient of $x^{p^r}$ on the left hand side is $0$ and we deduce that $[v_r]\circ b_{(0)}^{\c p^r}=0\pmod{[I_r]}$.
\end{proof}

\begin{lem}
\label{Ih*b0^p^k=0}
    Let $r\leq h$ be a nonnegative integer and suppose $k\geq \nu(r-1)$. If $\alpha$ is an element of the augmentation ideal of $K(n)_*\OSBPh{2k-2p^r+2}$ which maps to zero in $QK(n)_*(\OSBPh{2k-2p^r+2})/[I_{r}]$ then $\alpha\c b_{(0)}^{\c k}=0$.
\end{lem}

\begin{proof}
We work by induction on $r$. When $r=0$, the ideal $[I_r]$ is trivial and we need to check that $b_{(0)}\c \alpha = 0$ if $\alpha$ is decomposable. Say $\alpha=\beta\s \gamma$ where $\beta$ and $\gamma$ are in the augmentation ideal. Since $\Delta(b_{(0)})=[0_2]\otimes b_{(0)} + b_{(0)}\otimes [0_2]$, distributivity says $b_{(0)}\c (\beta\s \gamma) = ([0_2]\c \beta)\s (b_{(0)}\c \gamma) + (b_{(0)}\c \beta) \s ([0_2] \c \gamma) = 0$ since for any $y$, $y\c[0]=\eta(\epsilon(y))$.

Now suppose $r>0$ and $\alpha\in (\ker(\epsilon)^{\s 2}, [I_{r}])$. By induction we know that $(\beta\s \gamma)\c b_{(0)}^{\c k}=0$ and $[v_i]\c b_{(0)}^{\c k}=0$ for $i<r-1$, so it suffices to show this for $\alpha=[v_{r-1}]$.  We write
\[
[v_{r-1}]\c b_{(0)}^{\c k} =  \left([v_{r-1}]\c b_{(0)}^{\c p^{r-1}}\right) \c b_{(0)}^{\c (k - p^{r-1})}.
\]
By \Cref{bi^p^r}, $[v_{r-1}]\c b_{(0)}^{\c p^{r-1}}\in [I_{r-1}]$. Since $k-p^{r-1}\geq \nu(r-1)-p^{r-1}=\nu(r-2)$ the inductive hypothesis applies to show that
\[\left([v_{r-1}]\c b_{(0)}^{\c p^{r-1}}\right) \c b_{(0)}^{\c (k - p^{r-1})}=0.\]
\end{proof}

\begin{lem}
\label{vh*b1^p^h}
Let $r+2 \leq n$ be a nonnegative integer. In $QK(n)_*(\OSBPh{2*})/([I_{r}])$,
\[[v_r]\c b_{(1)}^{\c p^r} = \begin{cases}
-[v_{r+1}]\c b_{(0)}^{\c p^{r+1}} & r < h\\
0 & r = h.
\end{cases}\]
\end{lem}
\begin{proof}
We have
\[[p]_{K(n)}(x) = v_nx^{p^{n}} = O(x^{p^n}).\]
Thus, in $QK(n)_*(\OS{\BPn<h>}{2*})/([I_{r}])$ the Ravenel-Wilson relation says
\[0 = [v_r]\c b(x)^{\c p^r} + [v_{r+1}]\c b(x)^{\c p^{r+1}}+ O(x^{p^{r+2}}).\]
The $x^{p^{r+1}}$ coefficient on the right hand side is $[v_r]\c b_{(1)}^{\c p^r}+[v_{r+1}]\c b_{(0)}^{\c p^{r+1}}$, so $[v_r]\c b_{(1)}^{\c p^r} = -[v_{r+1}]\c b_{(0)}^{\c p^{r+1}}$ as claimed. When $r=h$, we have $v_{h+1}=0\in \BPn<h>$ so $[v_{h+1}]=0$.
\end{proof}

\begin{lem}
\label{x^p=vn*x-step}
Let $r\leq h$. In $K(n)_*(\OSBPh{2*})$,
\[[v_r]\c b_{(0)}^{\c p\nu(r-1)}\c b_{(1)}^{\c p^r} =
\begin{cases}
-[v_{r+1}]\c b_{(0)}^{\c p\nu(r)} & r<h\\
0 & r=h\\
\end{cases}\]
\end{lem}
\begin{proof}
Combine \Cref{vh*b1^p^h} and \Cref{Ih*b0^p^k=0}.
\end{proof}

\begin{proof}[Proof of \Cref{prop:xpzero}]
\Cref{x^p=vn*x-step} gives the following chain of equalities:
\begin{align*}
u^p = (b_{(0)}^{\c \nu(h)})^{\s p}
    &= [p]\c b_{(1)}^{\c \nu(h)}\\
    &=  -[v_1]\c b_{(0)}^{\c p} \c b_{(1)}^{\c (\nu(h)-1)} \\
    &=  +[v_2]\c b_{(0)}^{\c p\nu(1)}\c b_{(1)}^{\c (\nu(h)-\nu(1))}\\
    &\vdotswithin{=}\\[-3pt]
    &= \pm[v_{h-1}]\c b_{(0)}^{\c p\nu(h-2)} \c b_{(1)}^{\c (p^h+p^{h-1})}\\
    &= \mp[v_{h}]\c b_{(0)}^{\c p\nu(h-1)} \c b_{(1)}^{\c p^h}\\
    &=0.
\end{align*}
\end{proof}

\subsection{A Review of Hopf Rings and the Ravenel-Wilson relation} \label{sub:rwreview}
We list here some of the notation and formulas from \cite{Ravenel-Wilson} that we use in the above proof; the reader is referred to the original source for additional details.  

\begin{rec}
A \emph{Hopf ring} is a bigraded abelian group $A_{i,j}$ with the structures of
\begin{itemize}
\item coproduct $\Delta\colon A_{i,j} \to \bigoplus_{k + l = j} A_{i,k}\otimes A_{i,l}$,
\item $*$-product $*\colon A_{i,j}\otimes A_{i,k}\to A_{i,j+k}$, and
\item $\circ$-product $\circ\colon A_{i,j}\otimes A_{k,l}\to A_{i+k,j+l}$.
\end{itemize}
These satisfy various properties, including:
\begin{itemize}
\item The $*$-product and coproduct make $A_{i,*}$ into a commutative, cocommutative Hopf algebra with unit element denoted by $[0]_i$.
\item The $\circ$-product satisfies the distributivity relation
    \[a\circ (b*c) = \sum (a^{(1)}\circ c)*(a^{(2)}\circ c)\]
 and $a\circ [0]_i = \eta_{i+j}(\epsilon_j(a))$ where $\epsilon_j\colon A_{j,*}\to k$ is the augmentation map and $\eta_{i+j}\colon k\to A_{i+j,*}$ is the unit map.
\item There is a $\circ$-unit in $A_{0,0}$ denoted $[1]$.
\end{itemize}
For instance, there are elements $[n]=[1]*\cdots *[1]$ for $n$ a positive integer and we denote by $[-n]$ the antipode on $[n]$. The distributivity axiom implies that $[n]\circ x = \sum x^{(1)}*\cdots*x^{(n)}$.
\end{rec}

\begin{ntn*}
For any pair of rings $R$ and $S$ we may form the \emph{ring-ring} $R[S]$ by taking the group algebra of the additive group of $S$ and defining the $\circ$ product by $[s]\circ [s']=[ss']$.
\end{ntn*}

\begin{exm}
Given two homotopy commutative ring spectra $E$ and $F$, the bigraded abelian group $E_{*}\OS{F}{2*}$ forms a Hopf ring, provided there is a Kunneth isomorphism
\[E_*(\OS{F}{2i}\times\OS{F}{2j})\cong E_*(\OS{F}{2i})\otimes E_*(\OS{F}{2j}).\]
The diagonal map is induced by the diagonal of $\OS{F}{2i}$, the $*$-product is induced by the loop addition of $\OS{F}{2i}$, and the $\circ$-product is induced by the multiplication map $\susp^{i}F\sm \susp^{j}F\to \susp^{i+j}F$.
\end{exm}

\begin{ntn*}
In this situation, given an element $x\in\pi_n(F)$, we denote the Hurewicz image in $E_{0}\OS{F}{-n}$ by $[x]$. If we choose for $x$ the element $n$ in the image of the map $\Z\to \pi_0(F)$, then $[x]$ is the same $[n]$ we defined before. The elements $[x]$ satisfy $\Delta([x])=[x]\otimes [x]$, $[x]*[y]=[x+y]$, and $[x]\circ [y] = [xy]$. We get an injective map of Hopf rings from the ring-ring $E_*[\pi_*(F)]$ to $E_*\OS{F}{2*}$.
\end{ntn*}

Ravenel and Wilson study the case where both $E$ and $F$ are complex oriented.  In this situation, an additional relation, known as the \emph{Ravenel-Wilson relation}, is satisfied.
\begin{dfn}
Suppose $E$ and $F$ are complex oriented.  Then the complex orientation of $E$ gives an isomorphism $E_*(\CPinfty)\cong E_*\{\beta_{1},\beta_{2},\ldots\}$ where $|\beta_i|=2i$. Pushing these forward along the complex orientation map $f\colon \CPinfty \to \OS{F}{2}$ for $F$, we obtain elements
\[b_{i} := f_*(\beta_i) \in E_{2i}\OS{F}{2}.\]
We denote $b_{p^j}$ by $b_{(j)}$.
\end{dfn}

\begin{form*}[Ravenel-Wilson Relation]
Let $b(s)$ be the generating function $\sum b_i s^i$, and denote by $+_{E}$ the formal group addition of $E$, so $x +_{E} y = \sum a_{ij}^E x^iy^j$. Denote by $+_{[F]}$ the expression $\bigast_{i,j} [a_{ij}^F]\circ x^{\circ i}\circ y^{\circ j}$. The Ravenel-Wilson relation is
\[b(s+_{E}t) = b(s)+_{[F]}b(t).\]
This is an equality of generating functions, so it indicates that the coefficient of $s^{i}t^j$ on the left hand side is equal to the coefficient of the same monomial on the right hand side. We will only need to work out relations between the $b_{(j)} := b_{p^j}$, and so we will only need the following special case of this relation involving the $p$-series:
\[b([p]_{E}(s)) = [p]_{[F]}(b(s)).\]
\end{form*}

Ravenel and Wilson completely characterize the Hopf ring $E_{*}\OS{F}{2*}$ when $F$ is Landweber flat. They show that it is nearly free.
Recall that the free Hopf ring over $E_*[\pi_*(F)]$ on the $b_i$'s is formed by adjoining the elements $b_i$ with their copolynomial coproduct $\Delta(b_i) = \sum_{j+k=i} b_j\otimes b_k$, forming all symbols using $*$ and $\circ$ on the $b_i$'s, and then quotienting by an ideal to impose all of the Hopf ring axioms.
\begin{thm}[{\cite[Corollary 4.7]{Ravenel-Wilson}}]
When $E$ is complex oriented, the Hopf ring $E_{*}\OS{\BP}{2*}$ is the free Hopf ring on the $b_i$'s over the ring ring $E_*[\pi_*(\BP)]$ modulo the Ravenel-Wilson relation.
\end{thm}

\begin{rmk}
Note that $\BPh$ is not Landweber flat. The Hopf ring $E_*\OSBPh{2*}$ still contains the elements $b_{(i)}$ which still satisfy the Ravenel-Wilson relations, but it is neither true that the $b_{(i)}$ generate $E_*\OSBPh{2*}$ nor that the Ravenel-Wilson relations generate the relations in $E_*\OSBPh{2*}$.  However, the map $\BP\to \BPh$ induces a map of Hopf rings $E_*\OS{\BP}{2*}\to E_*\OSBPh{2*}$. Through $\circ$-degree $2\nu(h)$, this map is surjective with kernel $([v_{h+1}],[v_{h+2}],\ldots)$ and hence $E_*\OSBPh{2*}$ is isomorphic to the Ravenel-Wilson Hopf ring through $\circ$-degree $2\nu(h)$. This means that Ravenel and Wilson's work completely characterizes $E_{*}W_h$.
\end{rmk}

\appendix

%
%

\bibliographystyle{alpha}
\bibliography{bibliography}

\end{document}